\documentclass[final]{siamart1116}

\newcommand{\mx}{\mathbf{x}}
\newcommand{\me}{\mathbf{e}}
\newcommand{\ox}{\overline{x}}
\newcommand{\oV}{\overline{V}}

\newcommand{\tg}{\gamma}

\usepackage{dsfont}
\usepackage{bbm}

\newsiamremark{remark}{Remark}
\newsiamremark{assumption}{Assumption}

\usepackage{lipsum}
\usepackage{amsfonts}
\usepackage{graphicx}
\usepackage{epstopdf}
\usepackage{algorithmic}
\ifpdf
\DeclareGraphicsExtensions{.eps,.pdf,.png,.jpg}
\else
\DeclareGraphicsExtensions{.eps}
\fi

\usepackage{multirow}
\usepackage{booktabs}
\usepackage{verbatim}

\usepackage{subfigure}

\numberwithin{theorem}{section}

\newcommand{\TheTitle}{Swarming for Faster Convergence in Stochastic Optimization} 
\newcommand{\TheAuthors}{S. Pu and A. Garcia}

\headers{Swarming for Faster Convergence in Stochastic Optimization}{\TheAuthors}

\title{{\TheTitle}\thanks{Accepted in \emph{SIAM Journal on Control and Optimization}. The authors gratefully acknowledge partial support from AFOSR [FA9550-15-1-0504] and NSF [CMMI-1829552].
}}

\author{
	Shi Pu\thanks{School of Electrical, Computer and Energy Engineering, Arizona State University, Tempe, Arizona 85287-5706 USA
		(\email{sp3dw@virginia.edu}).}
	\and
	Alfredo Garcia\thanks{Department of Industrial and Systems Engineering, Texas A\&M University, College Station, TX 77843-3131 USA (\email{alfredo.garcia@tamu.edu})}
}

\ifpdf
\hypersetup{
  pdftitle={\TheTitle},
  pdfauthor={\TheAuthors}
}
\fi

\allowdisplaybreaks
\begin{document}

\maketitle

\begin{abstract}
  We study a distributed framework for stochastic optimization which is inspired by models of collective motion found in nature (e.g., swarming) with mild communication requirements.  Specifically, we analyze a scheme in which each one of $N > 1$ independent threads, implements in a distributed and unsynchronized fashion, a stochastic gradient-descent algorithm which is perturbed by a swarming potential. Assuming the overhead caused by synchronization is not negligible, we show the swarming-based approach exhibits better performance than a centralized algorithm (based upon the average of $N$ observations) in terms of (real-time) convergence speed.
  We also derive an error bound that is monotone decreasing in network size and connectivity.
  We characterize the scheme's finite-time performances for both convex and nonconvex objective functions.
\end{abstract}



\section{Introduction}

Consider the optimization problem%
\begin{equation}
\min_{x\in \mathbb{R}^{m}}f(x)  \label{opt Problem_def}
\end{equation}%
where $f:\mathbb{R}^{m}\rightarrow \mathbb{R}$ is a differentiable function.
In many applications, the gradient information $\nabla f(x)$ is not available in closed form and only noisy estimates can be computed by time-consuming simulations (see \cite{fu2015stochastic,spall2005introduction} for a survey of gradient
estimation techniques).
Noise can come from many sources such as incomplete convergence, modeling and discretization
error, and finite sample size for Monte-Carlo methods (see for instance \cite{kleijnen2008design}).

In this paper we consider the case in which $N>1$ independent and symmetric \emph{stochastic oracles} ($\mathcal{SO}$s) can be queried in order to obtain noisy gradient samples of the form $g(x,\xi)$ that satisfies the following condition:
\begin{assumption}
	\label{asp: gradient samples}
	Each random vector $\xi\in\mathbb{R}^p$ is independent. $\mathbb{E}[g(x,\xi)]=\nabla f(x)$, $\mathbb{E}[\|g(x,\xi)-\nabla f(x)\|^2]\le \sigma^2$ for some $\sigma>0$ for all $x\in\mathbb{R}^m$.
\end{assumption}
We assume the time needed to generate a gradient sample is random and not negligible (e.g., samples are obtained by executing time-consuming simulations).
In this setting, a centralized implementation of the stochastic gradient-descent algorithm in which a batch of $N$ samples is gathered at every iteration, incurs an overhead that is increasing in $N$. If sampling is undertaken in parallel by simultaneously querying $N$ $\mathcal{SO}$s, the overhead burden is reduced but is still non-negligible due to the need for synchronization. In this paper we study a distributed optimization scheme that does not require synchronization. In the scheme, $N>1$ independent computing threads implement each a stochastic gradient-descent algorithm in which a noisy gradient sample is perturbed by an attractive term, a function of the relative distance between solutions identified by neighboring threads (where the notion of neighborhood is related to a given network topology). This coupling is similar to those found in mathematical models of swarming (see \cite{gazi2011swarm}). Complex forms of collective motion such as swarms can be found in nature in many
organisms ranging from simple bacteria to mammals (see \cite{parrish2002self,okubo1986dynamical,reynolds1987flocks} for references). Such forms of collective behavior rely on limited communication amongst individuals and are believed to be effective for avoiding predators and/or for
increasing the chances of finding food (foraging) (see \cite{grunbaum1998schooling,punoise}). 

We show the proposed scheme has an important noise reduction property as noise realizations that induce individual trajectories differing too much from the group average are likely to be discarded because of the attractive term which aims to maintain cohesion.
In contrast to the centralized sample average approach, the noise reduction obtained in a swarming-based framework with $N>1$ threads does not require synchronization since each thread only needs the information on the current solutions identified by neighboring threads.
When sampling times are not negligible and exhibit large variation, synchronization may result in significant overhead so that the {\em real-time} performance of stochastic gradient-descent algorithms based upon the average of $N$ samples obtained in parallel is highly affected by large sampling time variability. In  contrast, the {\em real-time} performance of a swarming-based implementation with $N>1$ threads exhibits better performance as each thread can asynchronously update its solution based upon a  {\em small} sample size and {\em still} reap the benefits of noise reduction stemming from the swarming discipline.

The main contribution of this paper is the formalization of the benefits of the swarming-based approach for stochastic optimization. Specifically, we show the approach exhibits better performance than a centralized algorithm (based upon the average of $N$ observations) in terms of (real time) convergence speed.
We derive an error bound that is monotone decreasing in network size and connectivity. Finally, we characterize the scheme's finite-time performances for both convex and non-convex objective functions.

The structure of this paper is as follows. We introduce two candidate algorithms for solving problem (\ref{opt Problem_def}) in Section \ref{sec: pre} along with the related literature. In Section \ref{sec: pre anal} we perform preliminary analysis on the swarming-based approach. In Section \ref{sec: main} we present our main results. A numerical example is provided in Section \ref{sec: numeric}. We conclude the paper in Section \ref{sec: con}.

\section{Preliminaries}
\label{sec: pre}
We now present two algorithms for solving problem (\ref{opt Problem_def}). First, we introduce a centralized stochastic gradient-descent algorithm. Then we propose the corresponding swarming-based approach. 

\subsection*{A Centralized Algorithm}
A {\em centralized} gradient-descent algorithm is of the form (see for instance \cite{kushner2003stochastic}):%
\begin{equation}
x_{k+1}=x_{k}+\gamma_k u_k, \ \ k\in \mathbb{N} \label{eq:x_k centralized pre}
\end{equation}%
Suppose in a single step, $N>1$ $\mathcal{SO}$s each generate a noisy gradient sample $g(x_k,\xi_{i,k})$ in parallel, and $u_k=-(1/N)\sum_{i=1}^N g(x_k,\xi_{i,k})$. 
	Then (\ref{eq:x_k centralized pre}) can be rewritten as  
	\begin{equation}
	x_{k+1}=x_{k}+\gamma_k (-\nabla f(x_k)+\epsilon_k), \ \ k\in \mathbb{N} \label{eq:x_k centralized}
	\end{equation}
	where $\epsilon_k=\nabla f(x_k)-(1/N)\sum_{i=1}^{N}g(x_k,\xi_{i,k})$.
	$\gamma_k>0$ is the step size. In this paper we consider constant step size policies, i.e., $\gamma_k:=\gamma$ for some $\gamma>0$.

\subsection*{A Swarming-Based Approach}
A swarming-based asynchronous implementation has $N$ computing threads. In contrast to the centralized approach, each thread $i$ queries one $\mathcal{SO}$ and independently implements a gradient-descent algorithm based on only {\em one} sample per step:
\begin{equation}
x_{i,k+1}=x_{i,k}+\tg \left[-\nabla f(x_{i,k})+\varepsilon_{i,k}-a\sum_{j=1,j\neq
	i}^{N}\alpha_{ij}(x_{i,k}-x_{ij,k})\right], \ \ k\in \mathbb{N}
\label{opt eq: x_i,k_basic}
\end{equation}%
where $\varepsilon_{i,k}=\nabla f(x_{i,k})-g(x_{i,k},\xi_{i,k})$,
and $x_{ij,k}$ denotes the solution of thread $j$ at the time of thread $i$'s $k$-th update.

The last term on the right hand side of equation (\ref{opt eq: x_i,k_basic}) represents the function of mutual attraction between individual threads, in which $a>0$ measures the degree of attraction (see \cite{gazi2004class} for a reference). Let $\mathcal{G}=(\mathcal{V},\mathcal{E})$ denote the graph of all threads, where $\mathcal{V}$ stands for the set of vertices (threads), and $\mathcal{E}\subseteq \mathcal{V}\times\mathcal{V}$ is the set of edges connecting vertices.
Let $A=[\alpha_{ij}]\in \mathbb{R}^{N\times N}$ be the adjacency matrix of $\mathcal{G}$ where $\alpha_{ij}\in\{0,1\}$. $\alpha_{ij}=1$ indicates that thread $i$ is informed of the solution identified by threads $j$, or $(i,j)\in\mathcal{E}$. 
In this paper we assume the following condition regarding the network structure amongst threads:
\begin{assumption}
\label{asp: network}
The graph $\mathcal{G}$ corresponding to the network of threads is undirected ($A=A^T$) and connected, i.e., there is a path between every pair of vertices.
\end{assumption}

\subsection{Sampling Times and Time-Scales}
\label{subsec: sampling time}
In our comparison of convergence speeds in real-time we will need to describe the time-scales in which the algorithms described above operate. These time scales depend upon the following standing assumption regarding sampling times:
\begin{assumption}
	\label{asp: exponential}
	The times $\tilde{\Delta}t_i$ needed for generating gradient samples by oracles $\mathcal{SO}_i,i=1,\ldots,N$ are independent and exponentially distributed with mean $\Delta t>0$. 
\end{assumption}
It follows that for the centralized implementation, the time in between updates is $\tilde{\Delta} t_c=\max_{i \in \{1,\ldots,N\}}\{\tilde{\Delta} t_i\}$. 
To describe the time-scale for the swarming-based scheme, let $\tilde{\Delta} t_s=\min_{i \in \{1,\ldots,N\}}\{\tilde{\Delta} t_i\}$. In this time-scale, the time in between any two updates (by possibly different threads) is exponentially distributed with mean $\Delta t/N$.  Suppose there is a (virtual) global clock that ticks whenever a thread updates its solution,
and let $\mathds{1}_{i,k}$ be the indicator random variable for the event that thread $i$ provides the ($k+1$)-th update. 
We rewrite the asynchronous algorithm as follows:
	\begin{equation}
		x_{i,k+1}=x_{i,k}+\tg u_{i,k}\mathds{1}_{i,k}, \ \ k\in \mathbb{N}
			\label{opt eq: x_i,k}
	\end{equation}%
	where 
	\begin{equation*}
		u_{i,k}=-\nabla f(x_{i,k})+\varepsilon _{i,k}-a\sum_{j=1,j\neq
			i}^{N}\alpha_{ij}(x_{i,k}-x_{j,k}).
	\end{equation*}%
%

\subsection{Swarming for Faster Convergence: A Preview of the Main Results}
To illustrate the benefits of the swarming-based approach, we will compare the performances of the centralized approach (\ref{opt eq: x_i,k_basic}) and the swarming-based approach
(\ref{eq:x_k centralized}) in Section \ref{sec: main}. In what follows we provide a brief introduction to the main results.
	
	Denote $\ox_k:=(1/N)\sum_{i=1}^Nx_{i,k}$. Let $G_k:=\mathbb{E}[\|x_k-x^*\|^2]$ (respectively, $U_k:=\mathbb{E}[\|\ox_k-x^*\|^2]$) measure the quality of solutions under the centralized scheme (respectively, swarming-based approach). 
	
	Assuming $f(\cdot)$ is $\kappa$-strongly convex with Lipschitz continuous gradients, for sufficiently small $\gamma>0$,  we will show that $\limsup_k\mathbb{E}[\|x_k-x^*\|^2]$ and $\limsup_k\mathbb{E}[\|\ox_k-x^*\|^2]$ both have an upper bound in the order of $\mathcal{O}(\sigma^2/{(\kappa N)})$.
	
Regarding real-time performance, we will show that $\sup_k\mathbb{E}[U_{Nk}]$ and $\sup_k\mathbb{E}[G_k]$ converge at a similar rate. However, the time needed to complete $Nk$ iterations in the swarming-based method is approximately 
\begin{equation*}
Nk\cdot\Delta t/N=k\Delta t,
\end{equation*}
while the time needed for the centralized algorithm to complete $k$ steps is approximately $k\mathbb{E}[\tilde{\Delta} t_c]$.
The speed-up achieved through the swarming-based framework is due to the fact that $\Delta t<\mathbb{E}[\tilde{\Delta} t_c]=(1+1/2+\cdots+1/N)\Delta t$.
Since the relation holds true in general, this property is likely to be preserved under other distributions of sampling times as well.
As we shall see in Section \ref{sec: main}, the ratio of convergence speeds between the swarming-based approach and the centralized algorithm is approximately $\mathbb{E}[\tilde{\Delta} t_c]:\Delta t$. In other word, the convergence speed is inversely proportional to the rate that gradient samples are generated.

	\subsection{Literature Review}
	Our work is linked with the extensive literature in stochastic approximation (SA) methods dating back to \cite{robbins1951stochastic} and \cite{kiefer1952stochastic}. These works include the analysis of convergence (conditions and rates for convergence, proper step size choices) in the context of diverse noise models (see \cite{kushner2003stochastic}). 
	Recently there has been considerable interest in parallel or distributed implementation of stochastic gradient-descent methods (see \cite{cavalcante2013distributed,towfic2014adaptive,lobel2011distributed,srivastava2011distributed,wang2015cooperative} for examples). They mainly focus on minimizing a finite sum of convex functions: $f(x)=\sum_{i=1}^{N}f_i(x)$. Notice that we may write $f(x)=(1/N)\sum_{i=1}^{N}f(x)$, so that problem (\ref{opt Problem_def}) can be seen as a special case of the finite-sum formulation, and the swarming-based scheme (\ref{opt eq: x_i,k_basic}) resembles some of the algorithms in the literature (see \cite{srivastava2011distributed} for example).
	However, to the best of our knowledge, this literature does not address the noise reduction properties stemming from multi-agent coordination. Moreover, they do not consider random sampling times or the real-time performance of the algorithms.
	
	
	Our work is also related to population-based algorithms for simulation-based
	optimization. In these approaches, at every iteration, the quality of each
	solution in the population is assessed, and a new population of solutions is
	randomly generated based upon a given rule that is devised for achieving an
	acceptable trade-off between \textquotedblleft exploration\textquotedblright\
	and \textquotedblleft exploitation\textquotedblright. Recent efforts
	have focused on model-based approaches (see \cite{hu2007model}) which differ
	from population-based methods in that candidate solutions are generated
	at each round by sampling from a \textquotedblleft
	model\textquotedblright, i.e., a probability distribution over the space of
	solutions. The basic idea is to adjust the model based on the sampled
	solutions in order to bias the future search towards regions containing solutions of higher
	qualities (see \cite{hu2015model} for a recent survey). These methods are inherently
	{\em centralized} in that the updating of populations (or models) is performed
	after the quality of {\em all} candidate solutions is assessed.
	
	In \cite{pu2017flocking} we have also considered a swarming-type stochastic optimization method. In the paper we used stochastic differential equations to approximate the real-time optimization process. This approach relies on the assumption that step sizes are arbitrarily close to zero. Moreover, finite-time performance was obtained only for strongly convex functions.
	
	\section{Preliminary Analysis}
	\label{sec: pre anal}
	
	In this section we study the stochastic processes $\{x_{i,k}:k>0\}$ associated with each one of the $N>1$ threads in the swarming-based approach.
	The average solution $\bar{x}_{k}:=({1}/{N})\sum\nolimits_{i=1}^{N}x_{i,k}$ will play an important role in characterizing the performance. 
	This part of the analysis demonstrates the cohesiveness among solutions identified by different threads. To this end, we will analyze the process $\{\overline{V}_k:k>0\}$ defined as
	\[
	\overline{V}_k:=\frac{1}{N}\sum_{i=1}^{N}\Vert x_{i,k}-\bar{x}_k\Vert ^{2}.
	\]%
	 Let $e_{i,k}:=x_{i,k}-\bar{x}_{k}$ and $V_{i,k}:=\Vert e_{i,k}\Vert ^{2}$.
	 Then
	$ \overline{V}_k=({1}/{N})\sum_{i=1}^{N}V_{i,k}$. 
	\begin{lemma}
		 Under Algorithm (\ref{opt eq: x_i,k_basic}), suppose Assumption \ref{asp: exponential} holds.
			\begin{multline}
			\label{eq: V_{k+1}}
			\oV_{k+1}	= \oV_k-\frac{2\tg}{N}\sum_{i=1}^{N}\nabla^T f(x_{i,k})e_{i,k}\mathds{1}_{i,k}-\frac{2a\tg}{N}\sum_{i=1}^{N}\sum_{j=1,j\neq i}^N \alpha_{ij}e_{i,k}^T(e_{i,k}-e_{j,k})\mathds{1}_{i,k}\\
			+\frac{2\tg}{N}\sum_{i=1}^{N}e_{i,k}^T\varepsilon_{i,k}\mathds{1}_{i,k}
			 +\frac{\tg^2}{N}\sum_{i=1}^{N}\|\delta_{i,k}\|^2,
			\end{multline}
						where
									\begin{align}
										\label{delta}
									\delta_{i,k}=\delta_{i,k}^f+\delta_{i,k}^g+\delta_{i,k}^n,
									\end{align} 
								with
								\begin{equation}
								\label{delta^fgn}
												\left\{	\begin{aligned}
								\delta_{i,k}^f= & -\nabla f(x_{i,k})\mathds{1}_{i,k}+\frac{1}{N}\sum_{j=1}^{N}\nabla f(x_{j,k})\mathds{1}_{j,k},\\
								\delta_{i,k}^g= & -a\sum_{j=1,j\neq i}^N \alpha_{ij}(x_{i,k}-x_{j,k})\mathds{1}_{i,k}+\frac{a}{N}\sum_{j=1}^{N}\sum_{l=1,l\neq
									j}^{N}\alpha_{jl}(x_{j,k}-x_{l,k})\mathds{1}_{j,k}, \\
								\delta_{i,k}^n= & \varepsilon_{i,k}\mathds{1}_{i,k}-\frac{1}{N} \sum_{j=1}^{N} \varepsilon_{j,k}\mathds{1}_{j,k}.
								\end{aligned}\right.
								\end{equation}
								\label{lem1}
	\end{lemma}
	\begin{proof}
See Section \ref{subsec proof_lem1}.
	\end{proof}
		Let $\mathcal{L}=[l_{ij}]$ be	the Laplacian matrix associated with the adjacency matrix $A$, where $l_{ii}=\sum_j a_{ij}$ and $l_{ij}=-a_{ij}$ when $i\neq j$. For an undirected graph, the Laplacian matrix is symmetric positive semi-definite \cite{godsil2013algebraic}. 		Let $\me_k=[e_{1,k}^T,\ldots,e_{N,k}^T]^T$. 
		\begin{multline}
		\sum_{i=1}^{N}\sum_{j=1,j\neq
			i}^{N}\alpha_{ij}e_{i,k}^T(e_{i,k}-e_{j,k}) =-\sum_{i=1}^{N}\sum_{j=1}^{N}l_{ij}e_{i,k}^T(e_{i,k}-e_{j,k})
		=\sum_{i=1}^{N}\sum_{j=1}^{N}l_{ij}e_{j,k}^T e_{i,k} \\
		=\me_k^T(L\otimes I_m)\me_k\ge \lambda_2\me_k^T \me_k=\lambda_2\sum_{i=1}^{N}\|e_{i,k}\|^2,
		\label{eq: sum1}
		\end{multline}
		where $\lambda_2:=\lambda_2(\mathcal{L})$ is the second-smallest eigenvalue of $\mathcal{L}$, also called the {\em algebraic connectivity} of $\mathcal{G}$ (see \cite{godsil2013algebraic}).

	\begin{lemma}
			\label{lem2}
				Suppose Assumptions \ref{asp: gradient samples}, \ref{asp: network} and \ref{asp: exponential} hold. Under Algorithm (\ref{opt eq: x_i,k_basic}), let $\mathbf{x}_k =[x_{1,k},x_{2,k},\ldots,x_{N,k}]^T$, and let $\mathbb{E}[\oV_{k+1}|\mathbf{x}_k]$ denote the conditional expectation of $\oV_{k+1}$ given $\mathbf{x}_k$. Then,
			\begin{multline}
			\label{neq:doV}
			\mathbb{E}[\oV_{k+1}|\mathbf{x}_k] \le \oV_k-\frac{2\tg }{N^2}\sum_{i=1}^{N}\nabla^T f(x_{i,k})e_{i,k}-\frac{2}{N}a\lambda_2\tg \oV_k\\
			+\frac{\tg^2}{N^2}\sum_{i=1}^{N}\|\nabla f(x_{i,k})+a\sum_{j=1,j\neq
				i}^{N}\alpha_{ij}(x_{i,k}-x_{j,k})\|^2+\frac{\tg^2\sigma^2}{N}.
			\end{multline}
		\end{lemma}
	\begin{proof}
	See Appendix \ref{subsec proof_lem2}.
	\end{proof}
\begin{remark}
	Lemma \ref{lem2} sheds light on the cohesiveness property of solutions obtained by different threads. When $f(\cdot)$ satisfies proper regulatory conditions, $\oV_k$ is expected to decrease once exceeding a certain value. As a result, $\oV_k$ is bounded in expectation so that $x_{i,k}$ are not too different from each other.
	As we shall see in the next section, Lemma \ref{lem2} helps us characterize the superior performance of the swarming-based approach.
	\end{remark}

%

\section{Main Results}
\label{sec: main}

In this section we formalize the superior properties of the swarming-based framework. 
	The following additional assumptions will be used.
\begin{assumption}
	\label{asp:Lipschitz} {(Lipschitz) $\| \nabla f(x)-\nabla f(x^{\prime})\|\leq
		L\|x-x^{\prime}\|$ for some $L> 0$ and for all $x,x^{\prime}$}.
\end{assumption}
\begin{assumption}
	\label{asp:gradient_strconvexity} (Strong convexity) $(\nabla f (x) -\nabla
	f (x^{\prime}))^T (x-x^{\prime}) \geq \kappa \|x-x^{\prime}\|^{2}$ for some 
	$\kappa>0$ and for all $x,x^{\prime}$.
\end{assumption}
\begin{assumption}
	\label{asp:gradient_convexity} (Convexity) $f (x) \ge
	f (x^{\prime}) + \nabla^T f (x^{\prime})(x-x^{\prime})$ for all $x,x^{\prime}$.
\end{assumption}

Let us introduce a measure $U_k:=\Vert \bar{x}_k-x^{\ast }\Vert ^{2}$, of the distance between the average solution identified by all threads at step $k$ and an optimal solution $x^{\ast }$. We present some additional lemmas below.
\begin{lemma}
	Suppose Assumptions \ref{asp: gradient samples} and \ref{asp: exponential} hold. Under Algorithm (\ref{opt eq: x_i,k_basic}),
	\begin{multline}
		\label{eq: E U_{k+1}}
		\mathbb{E}[U_{k+1} | \mathbf{x}_k]
		\le U_k-\frac{2\tg}{N^2}\sum_{i=1}^{N}\nabla^T f(x_{i,k}) (x_{i,k}-x^{\ast })+\frac{2\tg}{N^2}\sum_{i=1}^{N}\nabla^T f(x_{i,k}) (x_{i,k}-\ox_k)\\
		+ \frac{\tg^2}{N^3}\sum_{i=1}^{N}\Vert\nabla f(x_{i,k})+a\sum_{j=1,j\neq
			i}^{N}\alpha_{ij}(x_{i,k}-x_{j,k})\Vert^2+\frac{\tg^2}{N^2}\sigma^2.
	\end{multline}
\label{lem eq: E U_{k+1}}
	\end{lemma}
\begin{proof}
See Appendix \ref{subsec proof_lem eq: E U_{k+1}}.
	\end{proof}

\begin{lemma}
		Let $\deg(i)$ denote the degree of vertex $i$ in graph $\mathcal{G}$ and let $\overline{d}:=\max_i \deg(i)$.
	\begin{align}
	\label{nabla+a}
	\sum_{i=1}^{N}\Vert\nabla f(x_{i,k})+a\sum_{j=1,j\neq
		i}^{N}\alpha_{ij}(x_{i,k}-x_{j,k})\Vert^2\le 2\sum_{i=1}^{N}\Vert\nabla f(x_{i,k})\Vert^2+8a^2N\overline{d}^2\oV_k.
	\end{align}
\label{lem nabla+a}
\end{lemma}
\begin{proof}
See Appendix \ref{subsec proof_lem nabla+a}.
\end{proof}

\begin{lemma}
	Suppose Assumptions \ref{asp: gradient samples}, \ref{asp: network}, \ref{asp: exponential}, \ref{asp:Lipschitz}, and \ref{asp:gradient_convexity} hold. Under Algorithm (\ref{opt eq: x_i,k_basic}), let $\omega\in(0,1)$ be arbitrary. Then,
		 \begin{multline}
		 \label{EUwV}
	     \mathbb{E}[U_{k+1}+\omega \oV_{k+1}| \mathbf{x}_k]
		 \le   U_k+\omega\oV_k-2\left[\frac{\tg}{N^2}-(1+\omega N)\frac{L\tg^2}{N^3}\right]\sum_{i=1}^{N}\nabla^T f(x_{i,k}) (x_{i,k}-x^{\ast })\\
		 +\frac{2}{N}(1-\omega)L\tg \oV_k-\frac{2}{N}\omega a\lambda_2\tg\oV_k+ \frac{8}{N^2}(1+\omega N) a^2\overline{d}^2\tg^2\oV_k
		 +(1+\omega N)\frac{\tg^2\sigma^2}{N^2}.
		 \end{multline}
	 \label{lem EUwV}
\end{lemma}
\begin{proof}
See Appendix \ref{subsec proof_lem EUwV}.
\end{proof}
	
\subsection{Strongly-Convex Objective Function}
In this section we will characterize and compare the performances of the swarming-based approach (\ref{opt eq: x_i,k_basic}) and the centralized method (\ref{eq:x_k centralized}) when the objective function $f(\cdot)$ is  strongly convex.
Specifically, we will discuss the ultimate error bounds and convergence speeds achieved under the two approaches. We will show that the two approaches are comparable in their ultimate error bounds while the swarming-based method enjoys a faster convergence.

The following theorem characterizes the performance of the swarming-based approach when the objective function is Lipschitz continuous and strongly convex.
	\begin{theorem}
	\label{thm: strongly convex constant}
	Suppose Assumptions  \ref{asp: gradient samples}, \ref{asp: network}, \ref{asp: exponential}, \ref{asp:Lipschitz}  and \ref{asp:gradient_strconvexity}  hold. Under Algorithm (\ref{opt eq: x_i,k_basic}) with step size satisfying
	\begin{equation}
	\gamma<\min\left(\frac{N}{(1+\hat{\omega} N)L},\frac{N}{2\kappa},\frac{N\lambda_2}{4a(N+1)\bar{d}^2}\right),
	\label{gamma bound strong convex}
	\end{equation}
we have
	\begin{multline}
	\label{EU_k thm1}
		\mathbb{E}[\Vert \bar{x}_k-x^{\ast }\Vert ^{2}]\le \phi_k=\frac{(1+\hat{\omega} N)\tg \sigma^2}{2\kappa N-2\kappa(1+\hat{\omega}N) L\gamma}\\
		+\left[U_0+\hat{\omega}\oV_0-\frac{(1+\hat{\omega} N)\tg \sigma^2}{2\kappa N-2\kappa(1+\hat{\omega}N) L\gamma}\right](1-C)^k.
	\end{multline}
Here $\hat{\omega}\in(0,1)$ is the solution to
\begin{multline}\label{hat_omega_equation}
	\kappa L\gamma\hat{\omega}^2-\left[\kappa+\frac{(N-1)}{N}\kappa L\gamma-L-a\lambda_2+4a^2\overline{d}^2 \tg\right]\hat{\omega}\\
	=-\kappa+\frac{\kappa L\gamma}{N}+L+\frac{4}{N}a^2\overline{d}^2\tg,
\end{multline}
and
\begin{equation}
C=\frac{2}{N}\kappa\tg-\frac{2}{N^2}\kappa(1+\hat{\omega}N) L\gamma^2.
\label{def_C}
\end{equation}
\end{theorem}
	\begin{proof}
See Appendix \ref{subsec proof_thm: strongly convex constant}.
	\end{proof}
\begin{remark}
	 There are two terms on the right hand side of inequality (\ref{EU_k thm1}). The first term is a positive constant, while
	the second term converges to $0$ exponentially fast. $C\in(0,1)$ determines the convergence speed. In the long run ($k\rightarrow\infty$), 
	\begin{align*}
	\limsup_k\mathbb{E}[\Vert \bar{x}_k-x^{\ast }\Vert ^{2}] \le \phi^*:= \frac{(1+\hat{\omega} N)\tg \sigma^2}{2\kappa N-2\kappa(1+\hat{\omega}N) L\gamma}.
	\end{align*}
		\end{remark}
The following corollary characterizes the dependency relationship between $\phi^*$ and several parameters including the step size $\tg$, algebraic connectivity $\lambda_2$, noise variance $\sigma^2$, and the convexity factor $\kappa$.
	\begin{corollary}
	Suppose all the conditions in Theorem \ref{thm: strongly convex constant} hold, and
	\begin{equation}
	\label{gamma bound2}
	\tg< \min\left(\frac{a\lambda_2+2L-2\kappa}{2(\kappa L+4a^2\overline{d}^2)}, \frac{2}{L}\right).
	\end{equation}
	Then,
	\begin{align*}
	\phi^*=\mathcal{O}\left(\frac{\sigma^2}{\kappa \lambda_2}\right).
	\end{align*}			
	If in addition $\lambda_2=\Theta(N)$, i.e., $\lambda_2$ is bounded above and below by $N$ asymptotically (e.g., when $\mathcal{G}$ is a complete graph),
	\begin{align*}
	\phi^*=\mathcal{O}\left(\frac{\sigma^2}{\kappa N}\right).
	\end{align*}
	\label{cor:bounds}
\end{corollary}
\begin{proof}
	By (\ref{gamma bound2}),
	\begin{equation*}
	\kappa L\tg+4a^2\overline{d}^2 \tg<\frac{1}{2}a\lambda_2+L-\kappa.
	\end{equation*}
	According to definition (\ref{hat_omega_equation}), 
	\begin{equation*}
	\frac{1}{2}a\lambda_2\hat{\omega}\le -\kappa+\frac{\kappa L\gamma}{N}+L+\frac{4}{N}a^2\overline{d}^2\tg\le -\kappa+L+\frac{1}{N}\left(\frac{1}{2}a\lambda_2+L-\kappa\right).
	\end{equation*}
	Hence $\hat{\omega}=\mathcal{O}(1/\lambda_2)$, and
	\begin{align*}
	\phi^*=\frac{(1+\hat{\omega} N)\tg \sigma^2}{2\kappa N-2\kappa(1+\hat{\omega}N) L\gamma}=\mathcal{O}\left(\frac{\sigma^2}{\kappa \lambda_2}\right).
	\end{align*}
	When $\lambda_2=\Theta (N)$, $\phi^*=\mathcal{O}({\sigma^2}/{(\kappa N)})$ follows immediately.
\end{proof}
\begin{remark}
	From Corollary \ref{cor:bounds}, $\phi^*$ is decreasing in the algebraic connectivity of the swarming network. In particular with a strong connectivity ($\lambda_2={\Theta}(N)$), $\phi^*$ decreases in the number of computing threads.
	This demonstrates the noise reduction property of the swarming-based scheme.
	\end{remark}

	\subsubsection{Comparison with Centralized Implementation}
	In this section, we compare the performances of the swarming-based approach (\ref{opt eq: x_i,k_basic}) and the centralized method (\ref{eq:x_k centralized}) in terms of their ultimate error bounds and convergence speeds.
	
	First we derive the convergence results for
	the centralized algorithm (\ref{eq:x_k centralized}):
	\begin{align*}
	x_{k+1}=x_k+\tg(-\nabla f(x_k)+\epsilon_k).
	\end{align*}
	Suppose Assumption \ref{asp:Lipschitz}, \ref{asp:gradient_strconvexity} hold and $\tg<2/L$.  Define $G_k :=\|x_k-x^*\|^2$ to measure the quality of solutions under this algorithm. Then,
	\begin{align*}
	G_{k+1}=&\|x_{k+1}-x^*\|^2\\
	= & \|x_k+\tg(-\nabla f(x_k)+\epsilon_k)-x^*\|^2\\
	= & G_k+\tg^2\|\nabla f(x_k)\|^2+\tg^2\|\epsilon_k\|^2
	-2\tg \nabla^T f(x_k)(x_k-x^*)+2\tg (x_k-x^*)^T \epsilon_k\\
	\le & G_k+L\tg^2 \nabla^T f(x_k)(x_k-x^*)
	-2\tg \nabla^T f(x_k)(x_k-x^*)+\tg^2\|\epsilon_k\|^2\\
	&+2\tg (x_k-x^*)^T \epsilon_k\\
	\le & G_k+\left(L\tg^2 -2\tg\right)\kappa G_k+\tg^2\|\epsilon_k\|^2+2\tg (x_k-x^*)^T \epsilon_k\\
	= & G_k+\kappa(L\tg-2)\tg G_k+\tg^2\|\epsilon_k\|^2+2\tg (x_k-x^*)^T \epsilon_k.
	\end{align*}
	Taking expectation on both sides,
		\begin{align*}
		\mathbb{E}[G_{k+1}]
		\le & \mathbb{E}[G_k]+\kappa(L\tg-2)\tg \mathbb{E}[G_k]+\frac{1}{N}\tg^2\sigma^2.
		\end{align*}
		We have
		\begin{align}
	\mathbb{E}[\|x_k-x^*\|^2]
		\le & \frac{\tg\sigma^2}{\kappa N(2-L\tg)}+\left[G_0-\frac{\tg\sigma^2}{\kappa N(2-L\tg)}\right](1-2\kappa\tg +\kappa L\tg^2)^{k-1}.
		\label{EG_k}
		\end{align}
		
		The long-run error bound of scheme (\ref{eq:x_k centralized}) is
		\begin{align*}
		\phi^{**}=\frac{\tg\sigma^2}{\kappa N(2-L\tg)}=\mathcal{O}\left(\frac{\sigma^2}{\kappa N}\right).
		\end{align*}
			By Corollary \ref{cor:bounds}, when $\lambda_2=\Theta(N)$ and $\tg$ is sufficiently small, $\phi^{**}$ and $\phi^{*}$ are in the same order.
		
		We now compare the real-time convergence speeds of the two approaches. In the swarming-based implementation, by (\ref{EU_k thm1}),
				\begin{align*}
					\mathbb{E}[\Vert \bar{x}_{Nk}-x^{\ast }\Vert ^{2}]\le \phi^*+ (U_0+\hat{\omega}\bar{V}_0-\phi^*)\left[1-\frac{2}{N}\kappa\tg+\frac{2}{N^2}\kappa(1+\hat{\omega}N) L\gamma^2\right]^{Nk}.
				\end{align*}
				Under a centralized algorithm, from (\ref{EG_k}),
						\begin{align*}
							\mathbb{E}[\|x_k-x^*\|^2]\le \phi^{**}+ (G_0-\phi^{**})(1-2 \kappa\tg+\kappa L\tg^2)^k.
						\end{align*}
					For large $N$,
					\begin{align*}
					\left[1-\frac{2}{N}\kappa\tg+\frac{2}{N^2}\kappa(1+\hat{\omega}N) L\gamma^2\right]^N=1-2\kappa\gamma+\mathcal{O}\left({\frac{1}{N^2}}\right)<1-2\kappa\tg +\kappa L\tg^2.
						\end{align*}
				We can see that $\mathbb{E}[U_{Nk}]$ converges at a similar rate to $\mathbb{E}[G_k]$. However, the time needed to complete $Nk$ iterations in the swarming-based method is approximately 
				\begin{equation*}
				Nk\cdot\Delta t/N=k\Delta t,
				\end{equation*}
				while the time needed for the centralized algorithm to complete $k$ steps is approximately $k\Delta t_c$, where (see \cite{Lugo})
				\begin{equation*}
				\Delta t_c=\mathbb{E}[\tilde{\Delta} t_c]=\mathbb{E}\left[\max_{i \in \{1,\ldots,N\}}\{\tilde{\Delta} t_i\}\right]=\left(1+\frac{1}{2}+\ldots+\frac{1}{N}\right)\Delta t>\Delta t.
				\end{equation*} 
	Clearly the swarming-based scheme converges faster than the centralized one in real time. The ratio of convergence speeds is approximately $\Delta t_c:\Delta t$, or $\ln N:1$.
		
\begin{remark}
	Another potential benchmark for assessing relative performance is the average solution of $N>1$ threads implementing independently the stochastic gradient-descent method (without an attraction term). However, this approach is deficient as the expected value of distance between the average solution and the optimal solution is bounded away from zero whenever the function is not symmetric with respect to the optimal solution.
	\end{remark}		
		
		In what follows we show the swarming-based framework could be applied to the optimization of general convex and nonconvex objective functions.
\subsection{General Convex Optimization}
\label{convex}
The following theorem characterizes the performance of the swarming-based approach for general convex objective functions. We refer to \cite{nemirovski2009robust} for utilizing the average of history iterates.
\begin{theorem}
	\label{thm: convex}
	Suppose Assumptions \ref{asp: gradient samples}, \ref{asp: network}, \ref{asp: exponential},  \ref{asp:Lipschitz} and \ref{asp:gradient_convexity} hold.  Define
	\begin{align}
		\tilde{\omega}:=\frac{NL+4a^2\overline{d}^2\tg}{NL+aN\lambda_2-4a^2N\overline{d}^2\tg}
		\label{omega_k}
	\end{align}
Let step size $\tg$ be such that $\tilde{\omega}\in(0,1)$ and 
	\begin{align}
		\mu=\frac{\tg}{N^2}-(1+\tilde{\omega} N)\frac{\tg^2 L}{N^3}>0.
		\label{mu_k}
	\end{align}
 Let $K>1$ and
 \begin{equation}
 \tilde{x}_K:=\frac{1}{K}\sum_{k=0}^{K-1}\bar{x}_k.
 \label{tilde x_k}
 \end{equation}	
	 Under Algorithm (\ref{opt eq: x_i,k_basic}),
	\begin{align}
	\mathbb{E}[f(\tilde{x}_K)-f(x^*)]
		\le \frac{1}{2NK\mu}\left[U_0+\tilde{\omega}\oV_0+
		(1+\tilde{\omega} N)\frac{K\tg^2\sigma^2}{N^2}\right].
		\label{convex_error}
	\end{align}
\end{theorem}
\begin{proof}
See Appendix \ref{subsec proof_thm: convex}.
\end{proof}
\begin{remark}
		To compute $\tilde{x}_K$ efficiently, notice that
	\begin{equation*}
	\tilde{x}_K=\frac{1}{N}\left(\frac{1}{K}\sum_{k=0}^{K-1}x_{i,k}\right)=\frac{1}{N}\tilde{x}_{i,K},
	\end{equation*}
	where $\tilde{x}_{i,K}=(1/K)\sum_{k=0}^{K-1}x_{i,k}$. At each step $k>0$, each thread $i$ needs only conduct local update $\tilde{x}_{i,k}=(1-1/k)\tilde{x}_{i,k-1}+(1/k)x_{i,k-1}$ to keep track of its own running average.
\end{remark}
It is clear that the upper bound is decreasing in $N$ and $\lambda_2$.
We now consider a possible strategy for choosing step size $\tg$.
\begin{corollary}
	\label{cor: convex}
	Suppose Assumptions \ref{asp: gradient samples}, \ref{asp: network}, \ref{asp: exponential},  \ref{asp:Lipschitz} and \ref{asp:gradient_convexity} hold.  Let step size $\tg$ and $K$ satisfy the conditions in Theorem \ref{thm: convex} and
	\begin{equation}
	\tg=\frac{D\sqrt{\lambda_2 N}}{\sigma\sqrt{K}}
	\le \min\left(\frac{\lambda_2}{8a\bar{d}^2}, \frac{(2L+a\lambda_2)N}{4(NL+L+a\lambda_2)L}\right)
	\label{gamma D}
	\end{equation}%
	for some $D>0$. Let $\tilde{\omega}$, $\mu$, and $\tilde{x}_K$ be set to (\ref{omega_k})-(\ref{tilde x_k}).
	Under Algorithm (\ref{opt eq: x_i,k_basic}),
\begin{align*}
		\mathbb{E}[f(\tilde{x}_K)-f(x^*)]
		\le \phi_K^*=\frac{\sigma\sqrt{N}}{D \sqrt{\lambda_2K}}\left[U_0+\tilde{\omega}\oV_0+
		\left(1+\frac{2NL+a\lambda_2}{2L+a\lambda_2}\right)\frac{D^2\lambda_2}{N}\right]
	\end{align*}
	\end{corollary}
\begin{proof}
See Appendix \ref{subsec proof_cor: convex}.
	\end{proof}
\begin{remark}
$\phi_K^*=\mathcal{O}(\sigma\sqrt{N}/\sqrt{\lambda_2 K})$.
If in addition $\lambda_2=\Theta(N)$,  then $\phi_K^*=\mathcal{O}(\sigma/\sqrt{K})$.
	Similar to the strongly-convex case, the error bound obtained here when $K=nk$ is comparable with that achieved by a centralized method with $k$ iterations (see \cite{nemirovski2009robust,ghadimi2013stochastic} for reference). Nevertheless, the time needed in a swarming-based approach is considerably less than that in a centralized one (the time ratio is the same as in the strongly-convex case).
	\end{remark}

\subsection{Nonconvex Optimization}
We further show that the swarming-based approach can be used for optimizing nonconvex objective functions. The employed randomization technique was introduced in \cite{ghadimi2013stochastic}.
\begin{theorem}
	\label{thm: nonconvex}
		Suppose Assumptions \ref{asp: gradient samples},  \ref{asp: network}, \ref{asp: exponential}, \ref{asp:Lipschitz} hold, and
		$a>5L/{(4\lambda_2)}$.
		Define
		\begin{equation}
		\check{\omega}:=\frac{NL+2(2L^2+4a^2\overline{d}^2)\tg}{4N(a\lambda_2-L)-4N(2L^2+4a^2\overline{d}^2)\tg}.
		\label{check_omega}
		\end{equation}
Let step size $\tg$ be such that $\check{\omega}\in(0,1)$ and 
	\begin{equation}
\check{\mu}=\frac{\tg}{2N^2}-\left(2+4\check{\omega}N\right)\frac{L\tg^2}{N^3}>0.
\label{check_mu}
\end{equation}
Let $K>1$. Define a new random variable $\ox_R$ as the following (see \cite{ghadimi2013stochastic}):
\begin{equation*}
\mathbb{P}[R=k]=\frac{1}{K},\forall k=0,1,\ldots,K-1.
\end{equation*} 
Then we have
	\begin{equation}
\frac{1}{L}\mathbb{E}[\|\nabla f(\ox_R)\|^2]
\le \frac{1}{NK\check{\mu}}\left[\frac{f(\ox_0)-f^*+\check{\omega}L\oV_0}{L}+\left(\frac{1}{2}+\check{\omega}N\right)\frac{K\tg^2\sigma^2}{N^2}\right],
\label{nonconvex_error}
\end{equation}
where $f^*$ denotes the optimal value of (\ref{opt Problem_def}).
\end{theorem}
\begin{proof}
See Appendix \ref{subsec proof_thm: nonconvex}.
\end{proof}		
\begin{remark}
	Note that the conditions and the obtained error bound in the nonconvex case are similar to those for optimizing general convex objective functions. Hence the discussions in Section \ref{convex} apply here.
	\end{remark}

\section{Numerical Example}
\label{sec: numeric}
In this section, we provide a numerical example to illustrate our theoretic findings. 
Consider the \emph{on-line} Ridge regression problem, i.e.,
\begin{equation}
\label{Ridge Regression}
\min_{x\in\mathbb{R}^d}f(x)\ \ \left(=\mathbb{E}_{u,v}\left[\left(u^T x-v\right)^2+\rho\|x\|^2\right]\right)
\end{equation}
where $\rho>0$ ia a penalty parameter.
Samples in the form of $(u,v)$ are gathered continuously with $u\in\mathbb{R}^d$ representing the features and $v\in\mathbb{R}$ being the observed outputs. We assume that each $u\in[-1,1]^d$ is uniformly distributed, and $v$ is drawn according to $v=u^T \tilde{x}+\varepsilon$. Here $\tilde{x}\in[0,1]^d$ is a predefined parameter, and $\varepsilon$ are independent Gaussian noises with mean $0$ and variance $1$. We further assume that the time needed to gather each sample is independent and exponentially distributed with mean $0.02s$.

Given a pair $(u,v)$, we can calculate an estimated gradient of $f(x)$:
\begin{equation}
g(x,u,v)=2(u^Tx -v)u+2\rho x.
\end{equation}
This is an unbiased estimator of $\nabla f(x)$ since 
\begin{equation*}
\nabla f(x)=\mathbb{E}_{u,v}\left[2uu^T x-2uv\right]+2\rho x=\mathbb{E}_{u,v}g(x,u,v).
\end{equation*}
Furthermore,
\begin{multline*}
\mathbb{E}_{u,v}\left[\|g(x,u,v)-\nabla f(x)\|^2\right]=\mathbb{E}_{u,v}\left[\left\|2uu^Tx-2uv-\mathbb{E}_{u,v}\left[2uu^T x-2uv\right]\right\|^2\right]\\
\le \mathbb{E}_{u,v}\left[\left\|2uu^Tx-\mathbb{E}_{u,v}[2uu^T] x-2u\left(u^T \tilde{x}+\varepsilon\right)+\mathbb{E}_{u,v}[2u\left(u^T \tilde{x}+\varepsilon\right)]\right\|^2\right]\\
=\mathbb{E}_{u,v}\left[\left\|\left(2uu^T-\mathbb{E}_{u,v}[2uu^T]\right) (x-\tilde{x})-2u\varepsilon\right\|^2\right]\le C_1\|x-\tilde{x}\|^2+C_2
\end{multline*}
for some $C_1,C_2>0$. As long as $\|x-\tilde{x}\|^2$ is bounded (which can be verified in the experiments), $\mathbb{E}_{u,v}\left[\|g(x,u,v)-\nabla f(x)\|^2\right]$ is uniformly bounded which satisfies Assumption \ref{asp: gradient samples}.
Notice that the Hessian matrix of $f(x)$ is $\mathbf{H}_f=\mathbb{E}_{u,v}\left[2uu^T\right]+2\rho I_d=(2/3+2\rho)I_d\succ 0$. Therefore $f(\cdot)$ is strongly convex, and problem (\ref{Ridge Regression}) has a unique solution $x^*$, which solves
\begin{equation*}
\nabla f(x^*)=\mathbb{E}_{u,v}\left[2uu^T\right] x^*-\mathbb{E}_{u,v}\left[2u\left(u^T \tilde{x}+\varepsilon\right)\right]+2\rho x^*=0.
\end{equation*}
We get
\begin{equation*}
x^*=\left(\mathbb{E}_{u,v}\left[uu^T\right]+\rho I_d\right)^{-1}\mathbb{E}_{u,v}\left[uu^T\right]\tilde{x}=(1+3\rho)^{-1}\tilde{x}.
\end{equation*}

In the experiments, we consider $9$ instances with different combinations of $d\in\{20,50,100\}$ and $N\in\{20,50,100\}$. For each instance, we run $100$ simulations with $x_{i,0}=x_0=\mathbf{0}$, and $\tilde{x}\in[0,1]^d$ is drawn uniformly randomly. Penalty parameter $\rho=0.1$ and step size $\gamma=0.01$. Under the swarming-based approach, we assume that $N$ threads constitute a random network, in which each two threads are linked with probability $p=10/N$. Attraction parameter is set to $a=1$. 

Figure \ref{fig: comparison} visualizes three sample simulations with $d=20$ and $N$ chosen from $\{20,50,100\}$ respectively. In all three cases, $\|\bar{x}_k-x^*\|^2$ (under the swarming-based approach) and $\|x_k-x^*\|^2$ (under the centralized scheme) both decrease linearly until reaching similar neighborhoods of $0$. It can be seen that the swarming-based method enjoys a faster convergence speed. As $N$ increases, the ultimate error bounds on $\mathbb{E}[\|\bar{x}_k-x^*\|^2]$ and $\mathbb{E}[\|x_k-x^*\|^2]$ both go down as expected. Furthermore, the centralized method slows down under larger $N$ because of increased synchronization overhead.
\begin{figure}
	\centering
		\subfigure[Instance $(d,N)=(20,20)$.]{\includegraphics[width=2.5in]{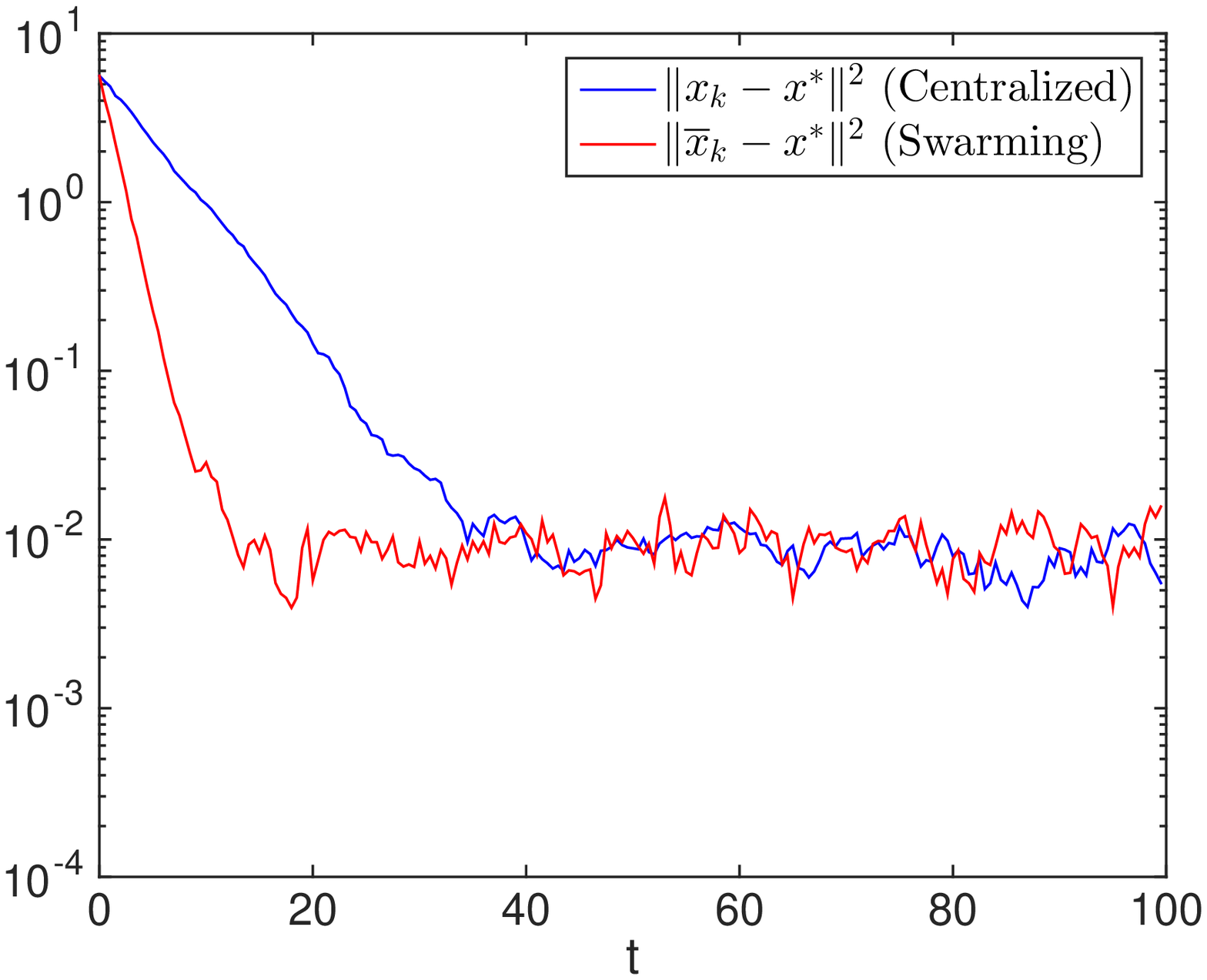}} 
	\subfigure[Instance $(d,N)=(20,50)$.]{\includegraphics[width=2.5in]{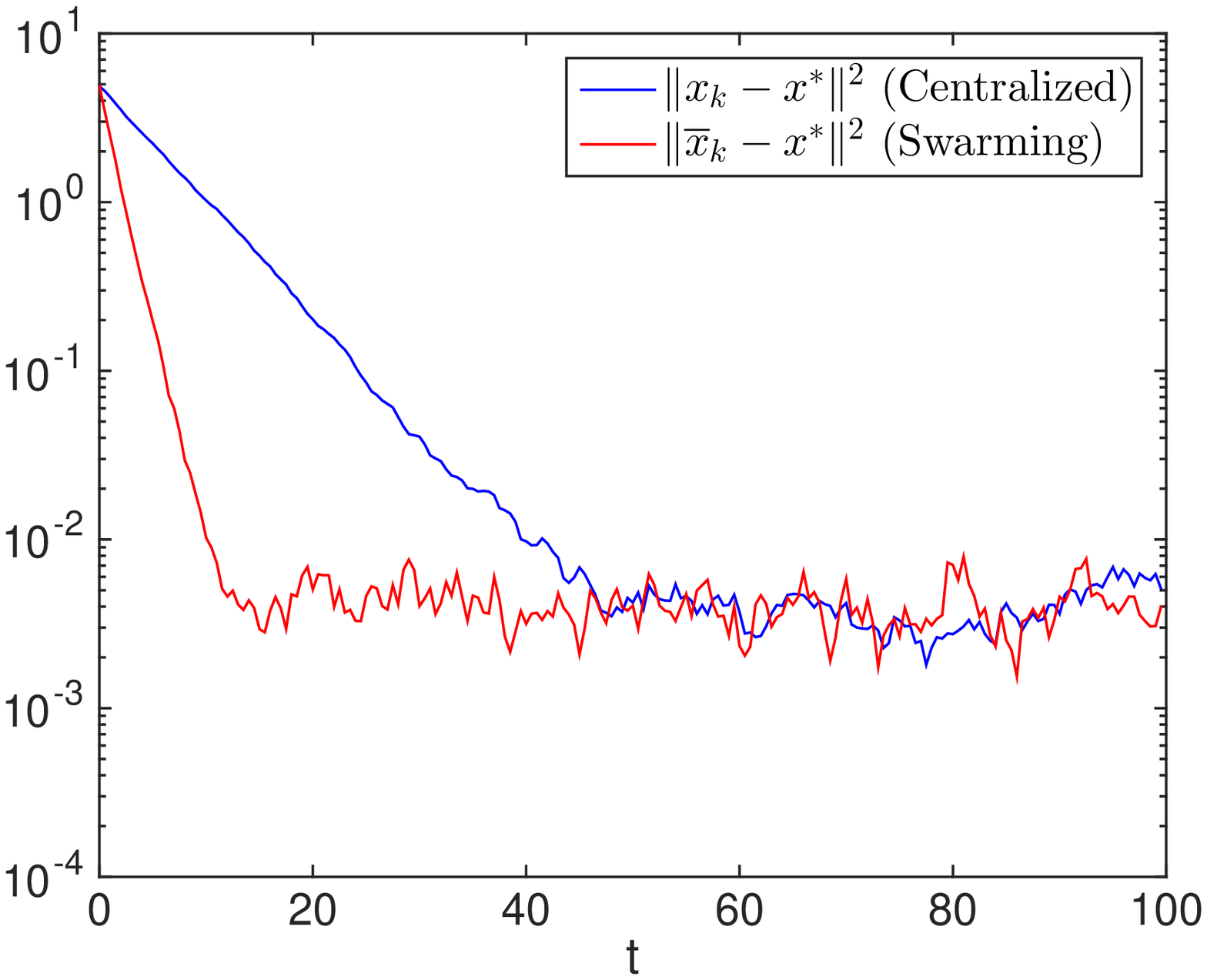}}
	\subfigure[Instance $(d,N)=(20,100)$.]{\includegraphics[width=2.5in]{{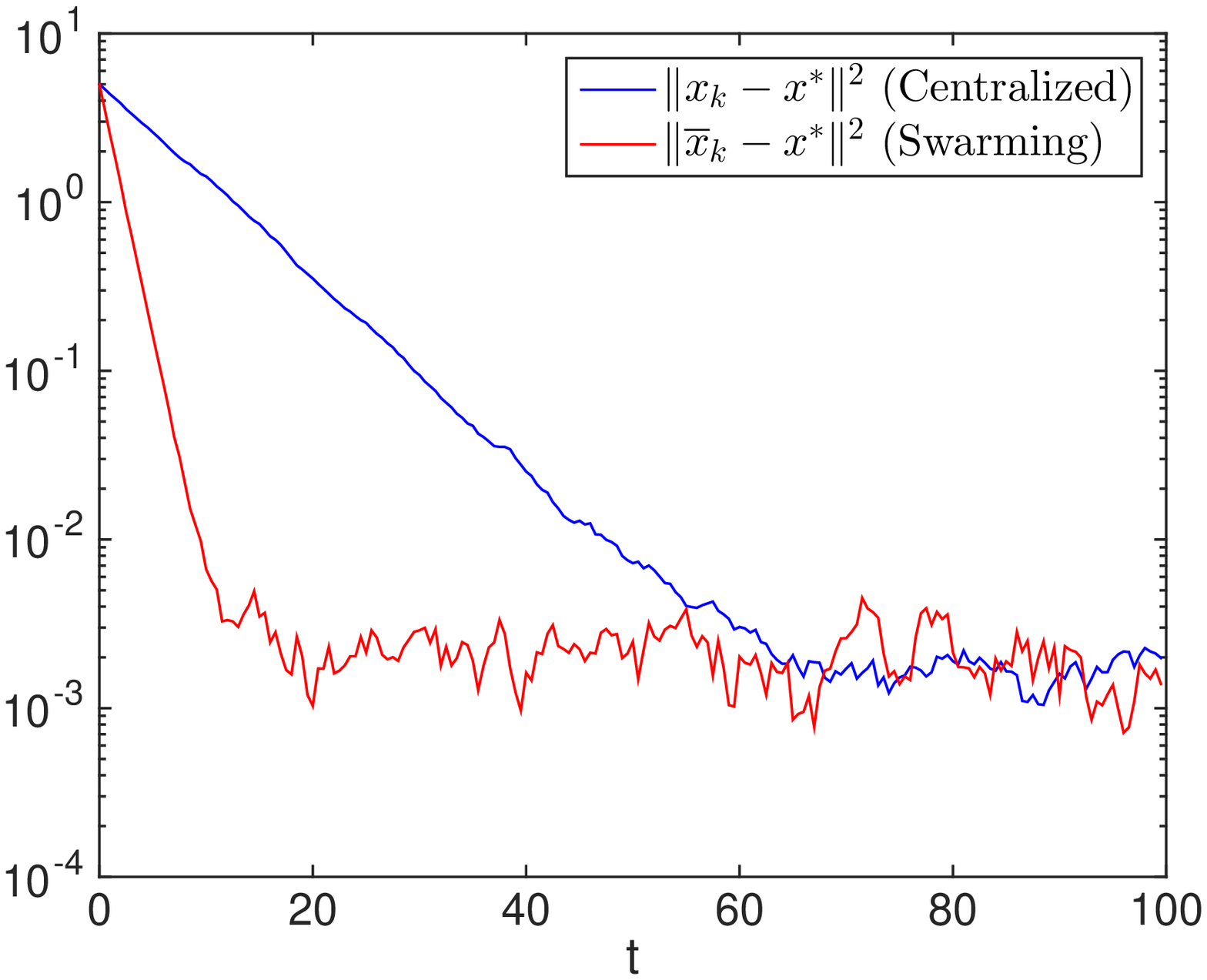}}} 
	\caption{Performance comparison between the swarming-based approach and the centralized scheme for on-line Ridge regression.}
	\label{fig: comparison}
\end{figure}

Table \ref{table: comparison} further compares the times it take for the two methods to meet certain stopping criteria. $T_s$ (respectively, $T_c$) denotes the average time over $100$ runs to reach $\|\bar{x}_k-x^*\|^2\le 0.1$ (respectively, $\|x_k-x^*\|^2\le 0.1$). We can see that $T_c>T_s$, and $T_c/T_s$ matches our theoretical estimate $\Delta t_c/\Delta t$ in all instances. 

In sum, the on-line Ridge regression example verified the advantages of the swarming-based approach when compared to a centralized scheme under synchronization overhead.
\begin{table}[!h]	
	\centering
	\caption{Comparison between the swarming-based approach and the centralized scheme.}
	\begin{tabular}{ccccc}
		Instance $(d,N)$ & $T_s /s$ & $T_c /s$ &  $T_c/T_s$  & $\Delta t_c/\Delta t$\\
		\midrule
		$(20,20)$ & 6.26 & 22.26 &  3.56 & 3.60\\
		$(20,50)$ & 6.77 & 30.14 & 4.45 & 4.50\\
		$(20,100)$ & 6.44 & 32.96 & 5.12 & 5.19\\
		$(50,20)$ & 7.89 & 28.11 & 3.56 & 3.60\\
		$(50,50)$ & 7.42 & 33.18 & 4.47 & 4.50\\
		$(50,100)$ & 7.35 & 37.70 & 5.13 & 5.19\\
		$(100,20)$ & 9.71 & 34.88 & 3.59 & 3.60\\
		$(100,50)$ & 8.79 & 39.10 & 4.45 & 4.50\\
		$(100,100)$ & 8.64 & 44.27 & 5.12 & 5.19\\
		\bottomrule
	\end{tabular}
\label{table: comparison}
\end{table}

\section{Conclusions}
\label{sec: con}
In recent years, the paradigm of cloud computing has emerged as an architecture for computing that makes use of distributed (networked) computing resources. However, distributing computational effort may not be an effective approach for stochastic optimization due to the effects of noise (which may be countered by averaging at the expense of overhead). Thus, in designing a distributed approach for stochastic optimization one faces a tradeoff between faster computing (since there is no need to synchronize threads) and the effects of noise.
 
 In this paper we have analyze a scheme for stochastic optimization that can be implemented in a cloud computing environment. In the scheme, $N>1$ independent computing threads implement each a stochastic gradient-descent algorithm in which a noisy gradient sample is perturbed by an attractive term, a function of the relative distance between solutions identified by neighboring threads where the notion of neighborhood is related to a pre-defined network topology. This coupling is similar to those found in mathematical models of swarming.
 
We consider the case in which the overhead caused by synchronization is not negligible. Such is the case for example when time-consuming simulations are needed to produce gradient estimates. We show the swarming-based approach exhibits better performance than a centralized algorithm (based upon the average of $N$ observations) in terms of real time convergence speed. We derive an error bound that is monotone decreasing in network size and connectivity. Finally, we characterize the scheme's finite-time performances for both convex and nonconvex objective functions.

\appendix
\section{Main Notations}

\begin{table}[!h]	
	\centering
	\caption{Main Notations}
	\begin{tabular}{cl}
		Symbol & Meaning \\
		\midrule
		$N$          & Number of computing threads \\
		$x^*$  & Optimal solution to problem (\ref{opt Problem_def})\\
		$x_k$ & Solution after the $k$-th step (centralized algorithm)\\
		$x_{i,k}$ & Solution of thread $i$ after the $k$-th step (swarming-based algorithm)\\
		$\bar{x}_{k}$  & $({1}/{N})\sum\nolimits_{i=1}^{N}x_{i,k}$\\
		$e_{i,k}$ & $x_{i,k}-\bar{x}_k$\\
		$V_{i,k}$ & $\|e_{i,k}\|^2$\\
				$\overline{V}_k$ & $({1}/{N})\sum_{i=1}^{N}V_{i,k}$\\
		$\gamma$ & Step size\\
		$\mathds{1}_{i,k}$ & Indicator random variable for the event that thread $i$ performs the \\
		& ($k+1$)-th  update (swarming-based scheme)\\
		$\varepsilon_k$& Simulation noise from one sample (centralized algorithm)\\
		$\varepsilon_{i,k}$ & Simulation noise from one sample (swarming-based algorithm)\\
		$\sigma$ & Variance of $\varepsilon_k$ and $\varepsilon_{i,k}$\\
		$a$ & Parameter for linear attraction in the swarming-based algorithm\\
		$\mathcal{G}$ & Interaction graph of ``swarming'' threads\\
		$A=[\alpha_{ij}]$ & Adjacency matrix of $\mathcal{G}$ \\
		$\mathcal{L}=[l_{ij}]$ & Laplacian matrix of $\mathcal{G}$\\
		$\lambda_2$ & Second-smallest eigenvalue of $\mathcal{L}$ (algebraic connectivity of $\mathcal{G}$)\\ 
		$\kappa$ & Strong convexity parameter\\
		$L$ & Lipschitz constant\\
		$\overline{d}$ & $\max_i \deg(i)$, where $\deg(i)$ denotes the degree of vertex $i$ in graph $\mathcal{G}$\\
		$\Delta t$ & Expected time for generating one noisy gradient sample.\\
		$\Delta t_c$ & Expected time to gather $N$ gradient samples (centralized algorithm).\\
		\bottomrule
	\end{tabular}
\end{table}

\section{Proofs of Preliminary Results}
\subsection{Proof of Lemma \ref{lem1}}
\label{subsec proof_lem1}
By relation (\ref{opt eq: x_i,k}), 
\begin{align}  \label{eq:macro}
	\ox_{k+1} 
	= & \frac{1}{N}\sum_{i=1}^{N}(x_{i,k}+\tg u_{i,k}\mathds{1}_{i,k}) \\
	= & \ox_k-\frac{\tg}{N}\sum_{i=1}^{N}\nabla f(x_{i,k})\mathds{1}_{i,k}-\frac{a\tg}{N}\sum_{i=1}^{N}\sum_{j=1,j\neq
		i}^{N}\alpha_{ij}(x_{i,k}-x_{j,k})\mathds{1}_{i,k}+\frac{\tg}{N} \sum_{i=1}^{N} \varepsilon_{i,k}\mathds{1}_{i,k}.\notag
\end{align}%
Hence
\begin{align*}
	\label{eq:micro}
	e_{i,k+1} = & x_{i,k+1}-\bar{x}_{k+1}   \\
	=& x_{i,k}-\ox_k+\tg \left[-\nabla f(x_{i,k})+\varepsilon_{i,k}-a\sum_{j=1,j\neq i}^N \alpha_{ij}(x_{i,k}-x_{j,k})\right]\mathds{1}_{i,k}\notag\\
	& +\frac{\tg}{N}\sum_{j=1}^{N}\nabla f(x_{j,k})\mathds{1}_{j,k} +\frac{a\tg}{N}\sum_{j=1}^{N}\sum_{l=1,l\neq
		j}^{N}\alpha_{jl}(x_{j,k}-x_{l,k})\mathds{1}_{j,k}	-\frac{\tg}{N} \sum_{i=1}^{N} \varepsilon_{j,k}\mathds{1}_{j,k} \notag\\
	= &e_{i,k}-\tg \nabla f(x_{i,k})\mathds{1}_{i,k}+\frac{\tg}{N}\sum_{j=1}^{N}\nabla f(x_{j,k})\mathds{1}_{j,k}-a\tg\sum_{j=1,j\neq i}^N \alpha_{ij}(x_{i,k}-x_{j,k})\mathds{1}_{i,k} \notag\\
	& +\frac{a\tg}{N}\sum_{j=1}^{N}\sum_{l=1,l\neq
		j}^{N}\alpha_{jl}(x_{j,k}-x_{l,k})\mathds{1}_{j,k}+\tg\varepsilon_{i,k}\mathds{1}_{i,k} -\frac{\tg}{N} \sum_{j=1}^{N} \varepsilon_{j,k}\mathds{1}_{j,k}\notag\\
	=& e_{i,k}+\tg\delta_{i,k},\notag
\end{align*}%
in which the last equality follows from the definition of $\delta_{i,k}$. We then have
\begin{multline*}
	{V_{i,k+1}}=  e_{i,k+1}^T e_{i,k+1}	= V_{i,k}+2\tg e_{i,k}^T\delta_{i,k}+\tg^2\|\delta_{i,k}\|^2\\
	=  V_{i,k}+2\tg e_{i,k}^T\delta_{i,k}^f+2\tg e_{i,k}^T\delta_{i,k}^g+2\tg e_{i,k}^T\delta_{i,k}^n+\tg^2\|\delta_{i,k}\|^2.
\end{multline*}%
Notice that
\begin{multline*}
	\sum_{i=1}^{N}e_{i,k}^T\delta_{i,k}^f = \sum_{i=1}^{N}e_{i,k}^T\left[-\nabla f(x_{i,k})\mathds{1}_{i,k}+\frac{1}{N}\sum_{j=1}^{N}\nabla f(x_{j,k})\mathds{1}_{j,k}\right]\\
	= -\sum_{i=1}^{N}\nabla^T f(x_{i,k})e_{i,k}\mathds{1}_{i,k}+\frac{1}{N}\sum_{i=1}^{N}e_{i,k}^T\sum_{j=1}^{N}\nabla f(x_{j,k})\mathds{1}_{j,k}	= -\sum_{i=1}^{N}\nabla^T f(x_{i,k})e_{i,k}\mathds{1}_{i,k},
\end{multline*}
	\begin{multline*}
	\sum_{i=1}^{N}e_{i,k}^T\delta_{i,k}^g=  -a\sum_{i=1}^{N}e_{i,k}^T \sum_{j=1,j\neq i}^N \alpha_{ij}(x_{i,k}-x_{j,k})\mathds{1}_{i,k}\\
	+\frac{a}{N}\sum_{i=1}^{N}e_{i,k}^T\sum_{j=1}^{N}\sum_{l=1,l\neq
		j}^{N}\alpha_{jl}(x_{j,k}-x_{l,k})\mathds{1}_{j,k}
	= -a\sum_{i=1}^{N}\sum_{j=1,j\neq i}^N \alpha_{ij}e_{i,k}^T(e_{i,k}-e_{j,k})\mathds{1}_{i,k},
\end{multline*}
and
\begin{align*}
	\sum_{i=1}^{N}e_{i,k}^T\delta_{i,k}^n= & \sum_{i=1}^{N}e_{i,k}^T\left(\varepsilon_{i,k}\mathds{1}_{i,k}-\frac{1}{N} \sum_{j=1}^{N} \varepsilon_{j,k}\mathds{1}_{j,k}\right)=\sum_{i=1}^{N}e_{i,k}^T\varepsilon_{i,k}\mathds{1}_{i,k}.
\end{align*}
We conclude that
\begin{align*}
	\oV_{k+1} = & \frac{1}{N}\sum_{i=1}^{N}V_{i,k+1} \\
	= &  \oV_k-\frac{2\tg}{N}\sum_{i=1}^{N}\nabla^T f(x_{i,k})e_{i,k}\mathds{1}_{i,k}-\frac{2a\tg}{N}\sum_{i=1}^{N}\sum_{j=1,j\neq i}^N \alpha_{ij}e_{i,k}^T(e_{i,k}-e_{j,k})\mathds{1}_{i,k}\\
	& +\frac{2\tg}{N}\sum_{i=1}^{N}e_{i,k}^T\varepsilon_{i,k}\mathds{1}_{i,k} +\frac{\tg^2}{N}\sum_{i=1}^{N}\|\delta_{i,k}\|^2.
\end{align*}

\subsection{Proof of Lemma \ref{lem2}}
\label{subsec proof_lem2}
Taking conditional expectation on both sides of (\ref{eq: V_{k+1}}) given $\mathbf{x}_k$,
\begin{multline*} 
	\mathbb{E}[\oV_{k+1}|\mathbf{x}_k] = \oV_k-\frac{2\tg }{N^2}\sum_{i=1}^{N}\nabla^T f(x_{i,k})e_{i,k}-\frac{2a \tg}{N^2}\sum_{i=1}^{N}\sum_{j=1,j\neq i}^N \alpha_{ij}e_{i,k}^T(e_{i,k}-e_{j,k})\\+\frac{\tg^2}{N}\sum_{i=1}^{N}\mathbb{E}[\|\delta_{i,k}\|^2 | \mathbf{x}_k].
\end{multline*}
In light of (\ref{eq: sum1}),
\begin{align*} 
	\mathbb{E}[\oV_{k+1}|\mathbf{x}_k] \le \oV_k-\frac{2\tg}{N^2}\sum_{i=1}^{N}\nabla^T f(x_{i,k})e_{i,k}-\frac{2}{N}a\lambda_2\tg\oV_k+\frac{\tg^2 }{N}\sum_{i=1}^{N}\mathbb{E}[\|\delta_{i,k}\|^2 | \mathbf{x}_k].
\end{align*}		
By (\ref{delta}) and the zero mean property of $\varepsilon_{i,k}$,
\begin{align*}
	\mathbb{E}[\|\delta_{i,k}\|^2 | \mathbf{x}_k]=\mathbb{E}[\|\delta_{i,k}^f+\delta_{i,k}^g+\delta_{i,k}^n\|^2 | \mathbf{x}_k]=\mathbb{E}[\|\delta_{i,k}^f+\delta_{i,k}^g\|^2 | \mathbf{x}_k]+\mathbb{E}[\|\delta_{i,k}^n\|^2 | \mathbf{x}_k].
\end{align*}
We now study the two terms $\mathbb{E}[\|\delta_{i,k}^f+\delta_{i,k}^g\|^2 | \mathbf{x}_k]$ and $\mathbb{E}[\|\delta_{i,k}^n\|^2 | \mathbf{x}_k]$, respectively.
First, notice that $\mathds{1}_{i,k}\mathds{1}_{j,k}= 0,\, \forall i\neq j$.
\begin{align*}
	&\mathbb{E}[\|\delta_{i,k}^f+\delta_{i,k}^g\|^2 | \mathbf{x}_k]\\
	=&\mathbb{E}\left[\|-\nabla f(x_{i,k})\mathds{1}_{i,k}+\frac{1}{N}\sum_{j=1}^{N}\nabla f(x_{j,k})\mathds{1}_{j,k}-a\sum_{j=1,j\neq i}^N \alpha_{ij}(x_{i,k}-x_{j,k})\mathds{1}_{i,k}\right.\\
	&\left.+\frac{a}{N}\sum_{j=1}^{N}\sum_{l=1,l\neq
		j}^{N}\alpha_{jl}(x_{j,k}-x_{l,k})\mathds{1}_{j,k}\|^2 | \mathbf{x}_k\right]\\
	=& \mathbb{E}\left[\|-\left(1-\frac{1}{N}\right)\nabla f(x_{i,k})\mathds{1}_{i,k}+\frac{1}{N}\sum_{j=1,j\neq i}^{N}\nabla f(x_{j,k})\mathds{1}_{j,k}\right.\\
	& -a\left(1-\frac{1}{N}\right)\sum_{j=1,j\neq i}^N \alpha_{ij}(x_{i,k}-x_{j,k})\mathds{1}_{i,k}\\
	&\left.+\frac{a}{N}\sum_{j=1,j\neq i}^{N}\sum_{l=1,l\neq
		j}^{N}\alpha_{jl}(x_{j,k}-x_{l,k})\mathds{1}_{j,k}\|^2 | \mathbf{x}_k\right]\\
	= & \mathbb{E}\left[\|-\left(1-\frac{1}{N}\right)\nabla f(x_{i,k})\mathds{1}_{i,k}-a\left(1-\frac{1}{N}\right)\sum_{j=1,j\neq i}^N \alpha_{ij}(x_{i,k}-x_{j,k})\mathds{1}_{i,k}\|^2 | \mathbf{x}_k\right]\\
	& +\sum_{j=1,j\neq i}^{N}\mathbb{E}\left[\|\frac{1}{N}\nabla f(x_{j,k})\mathds{1}_{j,k}+\frac{a}{N}\sum_{l=1,l\neq
		j}^{N}\alpha_{jl}(x_{j,k}-x_{l,k})\mathds{1}_{j,k}\|^2 | \mathbf{x}_k\right]\\
	= & \left(1-\frac{1}{N}\right)^2\frac{1}{N}\|\nabla f(x_{i,k})+a\sum_{j=1,j\neq i}^N \alpha_{ij}(x_{i,k}-x_{j,k})\|^2\\
	& +\frac{1}{N^3}\sum_{j=1,j\neq i}^{N}\|\nabla f(x_{j,k})+a\sum_{l=1,l\neq
		j}^{N}\alpha_{jl}(x_{j,k}-x_{l,k})\|^2.
\end{align*}
Second,
\begin{align*}
	\mathbb{E}[\|\delta_{i,k}^n\|^2 | \mathbf{x}_k]=&\mathbb{E}\left[\|\varepsilon_{i,k}\mathds{1}_{i,k}-\frac{1}{N} \sum_{j=1}^{N} \varepsilon_{j,k}\mathds{1}_{j,k}\|^2| \mathbf{x}_k\right]\\
	=&\mathbb{E}\left[\|\left(1-\frac{1}{N}\right)\varepsilon_{i,k}\mathds{1}_{i,k}-\frac{1}{N} \sum_{j=1,j\neq i}^{N} \varepsilon_{j,k}\mathds{1}_{j,k}\|^2| \mathbf{x}_k\right]\\
	=&\left(1-\frac{1}{N}\right)^2 \frac{1}{N}\mathbb{E}[\|\varepsilon_{i,k}\|^2 | \mathbf{x}_k]+\frac{1}{N^3}\sum_{j=1,j\neq i}^{N} \mathbb{E}[\|\varepsilon_{j,k}\|^2 | \mathbf{x}_k].
\end{align*}
It follows that
\begin{align*}
	\sum_{i=1}^{N}\mathbb{E}[\|\delta_{i,k}^f+\delta_{i,k}^g\|^2 | \mathbf{x}_k]= & \left(1-\frac{1}{N}\right)^2\frac{1}{N}\sum_{i=1}^{N}\|\nabla f(x_{i,k})+a\sum_{j=1,j\neq i}^N \alpha_{ij}(x_{i,k}-x_{j,k})\|^2\\
	& +\frac{1}{N^3}\sum_{i=1}^{N}\sum_{j=1,j\neq i}^{N}\|\nabla f(x_{j,k})+a\sum_{l=1,l\neq
		j}^{N}\alpha_{jl}(x_{j,k}-x_{l,k})\|^2\\
	=& \left(1-\frac{1}{N}\right)^2\frac{1}{N}\sum_{i=1}^{N}\|\nabla f(x_{i,k})+a\sum_{j=1,j\neq i}^N \alpha_{ij}(x_{i,k}-x_{j,k})\|^2\\
	& +\frac{(N-1)}{N^3}\sum_{i=1}^{N}\|\nabla f(x_{i,k})+a\sum_{j=1,j\neq
		i}^{N}\alpha_{ij}(x_{i,k}-x_{j,k})\|^2\\
	\le &\frac{1}{N}\sum_{i=1}^{N}\|\nabla f(x_{i,k})+a\sum_{j=1,j\neq
		i}^{N}\alpha_{ij}(x_{i,k}-x_{j,k})\|^2,
\end{align*}
and
\begin{align*}
	\sum_{i=1}^{N}\mathbb{E}[\|\delta_{i,k}^n\|^2 | \mathbf{x}_k]= & \left(1-\frac{1}{N}\right)^2 \frac{1}{N}\sum_{i=1}^{N}\mathbb{E}[\|\varepsilon_{i,k}\|^2 | \mathbf{x}_k]+\frac{1}{N^3}\sum_{i=1}^{N}\sum_{j=1,j\neq i}^{N} \mathbb{E}[\|\varepsilon_{j,k}\|^2 | \mathbf{x}_k]\\
	= &\frac{(N-1)}{N^2}\sum_{i=1}^{N}\mathbb{E}[\|\varepsilon_{i,k}\|^2 | \mathbf{x}_k]\le\sigma^2.
\end{align*}
In conclusion,
\begin{align*} 
	\mathbb{E}[\oV_{k+1}|\mathbf{x}_k] \le & \oV_k-\frac{2\tg}{N^2}\sum_{i=1}^{N}\nabla^T f(x_{i,k})e_{i,k}-\frac{2}{N}a\lambda_2\tg \oV_k+\frac{\tg^2 }{N}\sum_{i=1}^{N}\mathbb{E}[\|\delta_{i,k}\|^2 | \mathbf{x}_k]\\
	=& \oV_k-\frac{2\tg}{N^2}\sum_{i=1}^{N}\nabla^T f(x_{i,k})e_{i,k}-\frac{2}{N}a\lambda_2 \tg\oV_k\\
	& +\frac{\tg^2}{N^2}\sum_{i=1}^{N}\|\nabla f(x_{i,k})+a\sum_{j=1,j\neq
		i}^{N}\alpha_{ij}(x_{i,k}-x_{j,k})\|^2+\frac {\tg^2\sigma^2}{N}.
\end{align*}		

\subsection{Proof of Lemma \ref{lem eq: E U_{k+1}}}
\label{subsec proof_lem eq: E U_{k+1}}
	According to equation (\ref{eq:macro}), 
\begin{multline*}  
	\ox_{k+1} 
	=\ox_k-\frac{\tg}{N}\sum_{i=1}^{N}\nabla f(x_{i,k})\mathds{1}_{i,k}-\frac{a\tg}{N}\sum_{i=1}^{N}\sum_{j=1,j\neq
		i}^{N}\alpha_{ij}(x_{i,k}-x_{j,k})\mathds{1}_{i,k}\\
	+\frac{\tg}{N} \sum_{i=1}^{N} \varepsilon_{i,k}\mathds{1}_{i,k}.
\end{multline*}%
We have
\begin{align*}  \label{eq: U_{k+1}}
	& U_{k+1} \\
	= & \Vert\ox_k-x^{\ast }-\frac{\tg}{N}\sum_{i=1}^{N}\nabla f(x_{i,k})\mathds{1}_{i,k}-\frac{a\tg}{N}\sum_{i=1}^{N}\sum_{j=1,j\neq
		i}^{N}\alpha_{ij}(x_{i,k}-x_{j,k})\mathds{1}_{i,k}\notag \\
	&+ \frac{\tg}{N} \sum_{i=1}^{N} \varepsilon_{i,k}\mathds{1}_{i,k}\Vert^2 \notag\\
	=&  U_k-\frac{2\tg}{N}\sum_{i=1}^{N}\nabla^T f(x_{i,k}) (\ox_k-x^{\ast })\mathds{1}_{i,k}\notag\\
	&-\frac{2a\tg}{N}\sum_{i=1}^{N}\sum_{j=1,j\neq
		i}^{N}\alpha_{ij}(x_{i,k}-x_{j,k})^T(\ox_k-x^{\ast })\mathds{1}_{i,k}+\frac{2\tg}{N} \sum_{i=1}^{N} \varepsilon_{i,k}^T (\ox_k-x^{\ast })\mathds{1}_{i,k} \notag\\
	&+ \Vert\frac{\tg}{N}\sum_{i=1}^{N}\nabla f(x_{i,k})\mathds{1}_{i,k}+\frac{a\tg}{N}\sum_{i=1}^{N}\sum_{j=1,j\neq
		i}^{N}\alpha_{ij}(x_{i,k}-x_{j,k})\mathds{1}_{i,k}-\frac{\tg}{N} \sum_{i=1}^{N} \varepsilon_{i,k}\mathds{1}_{i,k}\Vert^2. \notag
\end{align*}
Taking conditional expectation on both sides,
\begin{align*}
	&\mathbb{E}[U_{k+1} | \mathbf{x}_k]\\
	=& U_k-\frac{2\tg}{N^2}\sum_{i=1}^{N}\nabla^T f(x_{i,k}) (\ox_k-x^{\ast })-\frac{2a\tg}{N^2}\sum_{i=1}^{N}\sum_{j=1,j\neq
		i}^{N}\alpha_{ij}(x_{i,k}-x_{j,k})^T(\ox_k-x^{\ast }) \notag\\
	&+ \mathbb{E}\left[\Vert\frac{\tg}{N}\sum_{i=1}^{N}\nabla f(x_{i,k})\mathds{1}_{i,k}+\frac{a\tg}{N}\sum_{i=1}^{N}\sum_{j=1,j\neq
		i}^{N}\alpha_{ij}(x_{i,k}-x_{j,k})\mathds{1}_{i,k} \right.\notag\\
	&\left.-\frac{\tg}{N} \sum_{i=1}^{N} \varepsilon_{i,k}\mathds{1}_{i,k}\Vert^2 | \mathbf{x}_k\right]\notag\\
	= & U_k-\frac{2\tg}{N^2}\sum_{i=1}^{N}\nabla^T f(x_{i,k}) (\ox_k-x^{\ast }) \notag\\
	&+ \frac{\tg^2}{N^3}\sum_{i=1}^{N}\mathbb{E}\left[\Vert \nabla f(x_{i,k})+a\sum_{j=1,j\neq
		i}^{N}\alpha_{ij}(x_{i,k}-x_{j,k}) -\varepsilon_{i,k}\Vert^2 | \mathbf{x}_k\right] \notag\\
	= & U_k-\frac{2\tg}{N^2}\sum_{i=1}^{N}\nabla^T f(x_{i,k}) (\ox_k-x^{\ast })\notag\\
	&+ \frac{\tg^2}{N^3}\sum_{i=1}^{N}\Vert\nabla f(x_{i,k})+a\sum_{j=1,j\neq
		i}^{N}\alpha_{ij}(x_{i,k}-x_{j,k})\Vert^2+\frac{\tg^2}{N^3}\sum_{i=1}^{N}\mathbb{E}[\Vert\varepsilon_{i,k}\Vert^2 |\mathbf{x}_k]\notag\\
	\le & U_k-\frac{2\tg}{N^2}\sum_{i=1}^{N}\nabla^T f(x_{i,k}) (x_{i,k}-x^{\ast })+\frac{2\tg}{N^2}\sum_{i=1}^{N}\nabla^T f(x_{i,k}) (x_{i,k}-\ox_k)\notag\\
	&+ \frac{\tg^2}{N^3}\sum_{i=1}^{N}\Vert\nabla f(x_{i,k})+a\sum_{j=1,j\neq
		i}^{N}\alpha_{ij}(x_{i,k}-x_{j,k})\Vert^2+\frac{\tg^2}{N^2}\sigma^2.\notag
\end{align*}

\subsection{Proof of Lemma \ref{lem nabla+a}}
\label{subsec proof_lem nabla+a}
	By the Cauchy-Schwarz inequality,\begin{multline*}
	\sum_{i=1}^{N}\Vert\nabla f(x_{i,k})+a\sum_{j=1,j\neq
		i}^{N}\alpha_{ij}(x_{i,k}-x_{j,k})\Vert^2\\
	\le 2\sum_{i=1}^{N}\Vert\nabla f(x_{i,k})\Vert^2+2a^2\sum_{i=1}^{N}\Vert \sum_{j=1,j\neq
		i}^{N}\alpha_{ij}(x_{i,k}-x_{j,k})\Vert^2,
\end{multline*}
and
\begin{multline*}
	\Vert \sum_{j=1,j\neq
		i}^{N}\alpha_{ij}(x_{i,k}-x_{j,k})\Vert^2\le \deg(i) \sum_{j=1,j\neq
		i}^{N}\alpha_{ij}\Vert x_{i,k}-x_{j,k}\Vert^2\\
	\le \overline{d} \sum_{j=1,j\neq
		i}^{N}\alpha_{ij}\Vert x_{i,k}-x_{j,k}\Vert^2.
\end{multline*}
It follows that
\begin{multline*}
	\sum_{i=1}^{N}\Vert \sum_{j=1,j\neq
		i}^{N}\alpha_{ij}(x_{i,k}-x_{j,k})\Vert^2\le \overline{d} \sum_{i=1}^{N}\sum_{j=1,j\neq
		i}^{N}\alpha_{ij}\Vert x_{i,k}-x_{j,k}\Vert^2\\
	=\overline{d} \sum_{i=1}^{N}\sum_{j=1}^{N}\alpha_{ij}\Vert e_{i,k}-e_{j,k}\Vert^2
	\le 2\overline{d} \sum_{i=1}^{N}\sum_{j=1}^{N}\alpha_{ij}(\Vert e_{i,k}\Vert^2+\Vert e_{j,k}\Vert^2)\\
	\le 4\overline{d}^2\sum_{i=1}^{N}\Vert e_{i,k}\Vert^2=4N\overline{d}^2\oV_k.
\end{multline*} 
Relation (\ref{nabla+a}) then follows.

\subsection{Proof of Lemma \ref{lem EUwV}}
\label{subsec proof_lem EUwV}
	 From (\ref{neq:doV}), (\ref{eq: E U_{k+1}}) and (\ref{nabla+a}),
\begin{align*}
	& \mathbb{E}[U_{k+1}+\omega \oV_{k+1} |\mathbf{x}_k] \\
	\le & U_k+\omega\oV_k-\frac{2\tg}{N^2}\sum_{i=1}^{N}\nabla^T f(x_{i,k}) (x_{i,k}-x^{\ast })+(1-\omega)\frac{2\tg}{N^2}\sum_{i=1}^{N}\nabla^T f(x_{i,k}) (x_{i,k}-\ox_k)\notag\\
	&-\frac{2}{N}\omega a\lambda_2 \tg\oV_k+ (1+\omega N)\frac{\tg^2}{N^3}\sum_{i=1}^{N}\Vert\nabla f(x_{i,k})+a\sum_{j=1,j\neq
		i}^{N}\alpha_{ij}(x_{i,k}-x_{j,k})\Vert^2\\
	&+(1+\omega N)\frac{\tg^2\sigma^2}{N^2}\\
	\le & U_k+\omega\oV_k-\frac{2\tg}{N^2}\sum_{i=1}^{N}\nabla^T f(x_{i,k}) (x_{i,k}-x^{\ast })+(1-\omega)\frac{2\tg}{N^2}\sum_{i=1}^{N}\nabla^T f(x_{i,k}) (x_{i,k}-\ox_k)\notag\\
	&-\frac{2}{N}\omega a\lambda_2 \tg\oV_k+ (1+\omega N)\frac{2\tg^2}{N^3}\left[\sum_{i=1}^{N}\Vert\nabla f(x_{i,k})\Vert^2+4a^2N\overline{d}^2\oV_k\right]\\
	&+(1+\omega N)\frac{\tg^2\sigma^2}{N^2}.
\end{align*}
By Assumptions \ref{asp:Lipschitz}, \ref{asp:gradient_convexity} and the fact $\nabla f(x^{\ast})=0$,
\begin{align}
	\label{neq: gradient}
	\sum_{i=1}^{N}\Vert\nabla f(x_{i,k})\Vert^2
	\le L\sum_{i=1}^{N}\nabla^T f(x_{i,k})(x_{i,k}-x^*),
\end{align}
and 
\begin{multline*}
	\frac{1}{N}\sum_{i=1}^N \nabla^T f(x_{i,k}) (x_{i,k}-\ox_k)	=\frac{1}{N}\sum_{i=1}^N
	\left(\nabla^T f(x_{i,k})-\nabla^T f(\overline{x}_k)\right) (x_{i,k}-%
	\overline{x}_k)\\
	\leq  \frac{1}{N}\sum_{i=1}^N L \|e_{i,k}\|^2\leq L\overline{V%
	}_k.
\end{multline*}
Hence,
\begin{align*}
	& \mathbb{E}[U_{k+1}+\omega \oV_{k+1}| \mathbf{x}_k]\\
	\le & U_k+\omega\oV_k-\frac{2\tg}{N^2}\sum_{i=1}^{N}\nabla^T f(x_{i,k}) (x_{i,k}-x^{\ast })+\frac{2}{N}(1-\omega) L\tg \oV_k-\frac{2}{N}\omega a\lambda_2\tg\oV_k\notag\\
	&+ (1+\omega N)\frac{2\tg^2}{N^3}\left[L\sum_{i=1}^{N}\nabla^T f(x_{i,k})(x_{i,k}-x^*)+4a^2N\overline{d}^2\oV_k\right]+(1+\omega N)\frac{\tg^2\sigma^2}{N^2}\notag\\
	= & U_k+\omega\oV_k-2\left[\frac{\tg}{N^2}-(1+\omega N)\frac{L\tg^2}{N^3}\right]\sum_{i=1}^{N}\nabla^T f(x_{i,k}) (x_{i,k}-x^{\ast })+\frac{2}{N}(1-\omega) L\tg \oV_k\notag\\
	&-\frac{2}{N}\omega a\lambda_2 \tg\oV_k+ \frac{8}{N^2}(1+\omega N) a^2\overline{d}^2\tg^2\oV_k+(1+\omega N)\frac{\tg^2\sigma^2}{N^2}.\notag
\end{align*}

\section{Proofs of Main Results}
\subsection{Proof of Theorem \ref{thm: strongly convex constant}}
\label{subsec proof_thm: strongly convex constant}
In light of Assumption \ref{asp:gradient_strconvexity},
\begin{multline*}
	\frac{1}{N}\sum_{i=1}^N \nabla^T f(x_{i,k}) (x_{i,k}-x^{\ast})= \frac{1}{N}%
	\sum_{i=1}^N (\nabla^T f(x_{i,k})-\nabla^T f(x^{\ast})) (x_{i,k}-x^{\ast})\\
	\geq\frac{\kappa}{N}\sum_{i=1}^{N}\Vert x_{i,k}-x^{\ast
		}\Vert ^{2}=\frac{\kappa}{N}\sum_{i=1}^{N}(x_{i,k}-\bar{x}_k+%
	\bar{x}_k-x^{\ast })^T (x_{i,k}-\bar{x}_k+\bar{x}_k-x^{\ast })\\
	=\frac{\kappa}{N}\sum_{i=1}^{N}\Vert x_{i,k}-\bar{x}_k\Vert ^{2}+%
	\frac{\kappa}{N}\sum_{i=1}^{N}\Vert \bar{x}_k-x^{\ast }\Vert ^{2}+\frac{2\kappa}{N}\sum_{i=1}^{N}(x_{i,k}-\bar{x}_{t})^T (\bar{x}_k-x^{\ast })
	=\kappa(\overline{V}_k+U_k).
\end{multline*}
Then by (\ref{EUwV}), when	$\gamma<N/{(1+\omega N)L}$,
\begin{align*}
	&\mathbb{E}[U_{k+1}+\omega \oV_{k+1}| \mathbf{x}_k] \\
	\le & U_k+\omega\oV_k-\frac{2}{N}\left[\tg -(1+\omega N)\frac{L\tg^2}{N}\right]\kappa (U_k+\overline{V}_k)+\frac{2}{N}(1-\omega) L\tg \oV_k\notag\\
	&-\frac{2}{N}\omega a \lambda_2\tg \oV_k+ \frac{8}{N^2}(1+\omega N) a^2\overline{d}^2\tg^2\oV_k+(1+\omega N)\frac{\tg^2\sigma^2}{N^2}\notag\\
	= & U_k+\omega\oV_k-\frac{2}{N}\kappa\left[1 -(1+\omega N)\frac{L\tg}{N}\right] \tg U_k+(1+\omega N)\frac{\tg^2\sigma^2}{N^2}\notag\\
	& - \frac{1}{N}\left[2\kappa-2\kappa(1+\omega N)\frac{L\gamma}{N}-2(1-\omega)L+2\omega a\lambda_2-\frac{8}{N}(1+\omega N)a^2\overline{d}^2 \tg\right]\tg\oV_k.\notag\\
\end{align*}
Since $\hat{\omega}$ satisfies (\ref{hat_omega_equation}),
\begin{multline*}
	2\kappa\hat{\omega}\left[1 -(1+\hat{\omega} N)\frac{L}{N}\tg\right] =2\kappa-2\kappa(1+\hat{\omega} N)\frac{L}{N}\gamma-2(1-\hat{\omega})L+2\hat{\omega} a\lambda_2\\
	-\frac{8}{N}(1+\hat{\omega} N)a^2\overline{d}^2 \tg.
\end{multline*}
%
Noting that $\hat{\omega}\in(0,1)$ from (\ref{gamma bound strong convex}) and (\ref{hat_omega_equation}), we have
\begin{multline*}
	\mathbb{E}[U_{k+1}+\hat{\omega} \oV_{k+1} |\mathbf{x}_k] \le  \left[1-\frac{2}{N}\kappa\tg+\frac{2}{N^2}\kappa(1+\hat{\omega}N) L\gamma^2 \right](U_k+\hat{\omega}\oV_k)\\
	+(1+\hat{\omega} N)\frac{\tg^2\sigma^2}{N^2}.
\end{multline*}
Taking full expectation on both sides,
\begin{multline*}
	\mathbb{E}[U_{k+1}+\hat{\omega} \oV_{k+1}] \le  \left[1-\frac{2}{N}\kappa\tg+\frac{2}{N^2}\kappa(1+\hat{\omega}N) L\gamma^2 \right]\mathbb{E}[U_k+\hat{\omega}\oV_k]\\
	+(1+\hat{\omega} N)\frac{\tg^2\sigma^2}{N^2}.
\end{multline*}
It follows that
\begin{multline*}
	\mathbb{E}[U_{k+1}+\hat{\omega} \oV_{k+1}]
	\le \frac{(1+\hat{\omega} N)\tg \sigma^2}{2\kappa N-2\kappa(1+\hat{\omega}N) L\gamma}\\
	+\left[U_0+\hat{\omega}\oV_0-\frac{(1+\hat{\omega} N)\tg \sigma^2}{2\kappa N-2\kappa(1+\hat{\omega}N) L\gamma}\right](1-C)^k.
\end{multline*}
Since $\hat{\omega} \oV_{k+1}\ge 0$, $\mathbb{E}[U_{k+1}] \le \mathbb{E}[U_{k+1}+\hat{\omega} \oV_{k+1}]$, the desired result follows.

\subsection{Proof of Theorem \ref{thm: convex}}
\label{subsec proof_thm: convex}
	Suppose Assumption \ref{asp:gradient_convexity} holds. By (\ref{EUwV}),
\begin{multline*}
	2\left[\frac{\tg}{N^2}-(1+\tilde{\omega} N)\frac{L\tg^2}{N^3}\right]\sum_{i=1}^{N}\nabla^T f(x_{i,k}) (x_{i,k}-x^{\ast })\\
	\le   U_k+\tilde{\omega}\oV_k-\mathbb{E}[U_{k+1}+\tilde{\omega} \oV_{k+1}| \mathbf{x}_k]
	+\frac{2}{N}(1-\tilde{\omega}) L\tg\oV_k-\frac{2}{N}\tilde{\omega} a\lambda_2 \tg\oV_k\\
	+ \frac{8}{N^2}(1+\tilde{\omega} N) a^2\overline{d}^2\tg^2\oV_k+(1+\tilde{\omega} N)\frac{\tg^2\sigma^2}{N^2}.
\end{multline*}
By the convexity of $f(\cdot)$, $\nabla^T f(x_{i,k}) (x_{i,k}-x^{\ast })\ge f(x_{i,k})-f(x^*)$.
From (\ref{omega_k}), $\tilde{\omega}$ satisfies
\begin{align*}
	2(1-\tilde{\omega})L-2\tilde{\omega} a\lambda_2+ \frac{8}{N}(1+\tilde{\omega} N) a^2\overline{d}^2\tg=0.
\end{align*}
We have
\begin{multline*}
	2\left[\frac{\tg}{N^2}-(1+\tilde{\omega} N)\frac{L\tg^2}{N^3}\right]\sum_{i=1}^{N}[f(x_{i,k})-f(x^*)]\\
	\le   U_k+\tilde{\omega}\oV_k-\mathbb{E}[U_{k+1}+\tilde{\omega} \oV_{k+1}| \mathbf{x}_k]+
	(1+\tilde{\omega} N)\frac{\tg^2\sigma^2}{N^2}.
\end{multline*}			
Taking full expectation on both sides and sum from $k=0$ to $k=K-1$,
\begin{multline*}
	2\sum_{k=0}^{K-1}\left[\frac{\tg}{N^2}-(1+\tilde{\omega} N)\frac{L\tg^2}{N^3}\right]\sum_{i=1}^{N}\mathbb{E}[f(x_{i,k})-f(x^*)]\\
	\le   U_0+\tilde{\omega}\oV_0-\mathbb{E}[U_{K}+\tilde{\omega} \oV_K]+
	(1+\tilde{\omega} N)\frac{K\tg^2\sigma^2}{N^2}\\
	\le U_0+\tilde{\omega}\oV_0+(1+\tilde{\omega} N)\frac{K\tg^2\sigma^2}{N^2}.
\end{multline*}		
In light of (\ref{mu_k})
and noting that $f(\bar{x}_k)\le (1/N)\sum_{i=1}^{N}f(x_{i,k})$ by the convexity of $f(\cdot)$,
\begin{align*}
	\frac{1}{K}\sum_{k=0}^{K-1} \mathbb{E}[f(\bar{x}_k)-f(x^*)]
	\le \frac{1}{2NK\mu}\left[U_0+\tilde{\omega}\oV_0+
	(1+\tilde{\omega} N)\frac{K\tg^2\sigma^2}{N^2}\right].
\end{align*}	
Since
$
	\tilde{x}_K=(1/K)\sum_{k=0}^{K-1}\bar{x}_k,
$
\begin{align*}
	\mathbb{E}[f(\tilde{x}_K)-f(x^*)]
	\le \frac{1}{2NK\mu}\left[U_0+\tilde{\omega}\oV_0+
	(1+\tilde{\omega} N)\frac{K\tg^2\sigma^2}{N^2}\right].
\end{align*}

\subsection{Proof of Corollary \ref{cor: convex}}
\label{subsec proof_cor: convex}
By (\ref{omega_k}) and the first inequality relation in (\ref{gamma D}), 
\begin{equation}
	\tilde{\omega}=\frac{NL+4a^2\overline{d}^2\tg}{NL+aN\lambda_2-4a^2N\overline{d}^2\tg}\le\frac{2NL+a\lambda_2}{2NL+aN\lambda_2}.
	\label{tilde_omega_inequality}
\end{equation}
Then,
\begin{multline*}
	\mu=\frac{\tg}{N^2}-(1+\tilde{\omega} N)\frac{L\tg^2}{N^3}=\frac{\tg}{N^2}\left[1-(1+\tilde{\omega} N)\frac{L\tg}{N}\right]\\
	\ge \frac{\tg}{N^2}\left[1-\left(1+\frac{2NL+a\lambda_2}{2L+a\lambda_2}\right)\frac{L\tg}{N}\right]
	\ge \frac{\tg}{2N^2},
\end{multline*}
where the last step follows from the second inequality relation in (\ref{gamma D}).
In light of Theorem \ref{thm: convex} and (\ref{tilde_omega_inequality}),
\begin{align*}
	\mathbb{E}[f(\tilde{x}_K)-f(x^*)]
	\le &\frac{1}{2NK\mu}\left[U_0+\tilde{\omega}\oV_0+
	(1+\tilde{\omega} N)\frac{K\tg^2\sigma^2}{N^2}\right]
	\\
	\le & \frac{N}{\tg K}\left[U_0+\tilde{\omega}\oV_0+
	\left(1+\frac{2NL+a\lambda_2}{2L+a\lambda_2}\right)\frac{K\tg^2\sigma^2}{N^2}\right]\\
	=& \frac{\sigma\sqrt{N}}{D \sqrt{\lambda_2K}}\left[U_0+\tilde{\omega}\oV_0+
	\left(1+\frac{2NL+a\lambda_2}{2L+a\lambda_2}\right)\frac{D^2\lambda_2}{N}\right].
\end{align*}

\subsection{Proof of Theorem \ref{thm: nonconvex}}
\label{subsec proof_thm: nonconvex}
	By Assumption \ref{asp:Lipschitz} and (\ref{eq:macro}), we have
	\begin{align*}
	& f(\ox_{k+1}) \\
	\le & f(\ox_k)+\nabla^T f(\ox_k)(\ox_{k+1}-\ox_k)+\frac{L}{2}\|\ox_{k+1}-\ox_k\|^2\\
	= & f(\ox_k)-\frac{\tg}{N}\sum_{i=1}^{N}\nabla^T f(\ox_k) \nabla f(x_{i,k})\mathds{1}_{i,k}-\frac{a\tg}{N}\sum_{i=1}^{N}\sum_{j=1,j\neq
		i}^{N}\alpha_{ij}\nabla^T f(\ox_k)(x_{i,k}-x_{j,k})\mathds{1}_{i,k}\\
	&+\frac{\tg}{N} \sum_{i=1}^{N} \nabla^T f(\ox_k)\varepsilon_{i,k}\mathds{1}_{i,k} \\
	&+ \frac{L}{2}\Vert\frac{\tg}{N}\sum_{i=1}^{N}\nabla f(x_{i,k})\mathds{1}_{i,k}+\frac{a\tg}{N}\sum_{i=1}^{N}\sum_{j=1,j\neq
		i}^{N}\alpha_{ij}(x_{i,k}-x_{j,k})\mathds{1}_{i,k}-\frac{\tg}{N} \sum_{i=1}^{N} \varepsilon_{i,k}\mathds{1}_{i,k}\Vert^2.
	\end{align*}
	Taking conditional expectation on both sides,
	\begin{align}
	\label{E f(ox_k+1) pre}
	& \mathbb{E}[f(\ox_{k+1})\mid \mx_k] \\
	\le & f(\ox_k)-\frac{\tg}{N^2}\sum_{i=1}^{N}\nabla^T f(\ox_k) \nabla f(x_{i,k})-\frac{a\tg}{N^2}\sum_{i=1}^{N}\sum_{j=1,j\neq
		i}^{N}\alpha_{ij}\nabla^T f(\ox_k)(x_{i,k}-x_{j,k}) \notag\\
	&+ \frac{L\tg^2}{2N^3}\sum_{i=1}^{N}\mathbb{E}\left[\Vert\nabla f(x_{i,k})+a\sum_{j=1,j\neq
		i}^{N}\alpha_{ij}(x_{i,k}-x_{j,k})-\varepsilon_{i,k}\Vert^2\mid \mx_k\right] \notag\\
	= & f(\ox_k)-\frac{\tg}{N}\|\nabla f(\ox_k)\|^2-\frac{\tg}{N^2}\sum_{i=1}^{N}\nabla^T f(\ox_k) [\nabla f(x_{i,k})-\nabla f(\ox_k)] \notag\\
	& +\frac{L\tg^2}{2N^3}\sum_{i=1}^{N}\Vert\nabla f(x_{i,k})+a\sum_{j=1,j\neq
		i}^{N}\alpha_{ij}(x_{i,k}-x_{j,k})\Vert^2+\frac{L\tg^2}{2N^3}\sum_{i=1}^{N}\mathbb{E}[\|\varepsilon_{i,k}\|^2\mid \mx_k]. \notag
	\end{align}
	Note that by Lemma \ref{lem nabla+a} and Assumption \ref{asp:Lipschitz} we have 
	\begin{multline}
	\label{nabla+a_2}
	\sum_{i=1}^{N}\Vert\nabla f(x_{i,k})+a\sum_{j=1,j\neq i}^{N}\alpha_{ij}(x_{i,k}-x_{j,k})\Vert^2\le 2\sum_{i=1}^{N}\Vert\nabla f(x_{i,k})\Vert^2+8a^2N\overline{d}^2\oV_k\\
	\le 4N\Vert\nabla f(\ox_k)\Vert^2+4\sum_{i=1}^{N}\Vert\nabla f(x_{i,k})-\nabla f(\ox_k)\Vert^2+8a^2N\overline{d}^2\oV_k\\
	\le 4N\Vert\nabla f(\ox_k)\Vert^2+(4L^2 N+8a^2N\overline{d}^2)\oV_k.
	\end{multline}
	In light of Assumption \ref{asp:Lipschitz},
	\begin{multline*}
	-\sum_{i=1}^{N}\nabla^T f(\ox_k) [\nabla f(x_{i,k})-\nabla f(\ox_k)]\le
	L\sum_{i=1}^{N}\|\nabla f(\ox_k)\|\|e_{i,k}\|\\
	\le \frac{1}{2}\sum_{i=1}^{N}(\|\nabla f(\ox_k)\|^2+L^2\|e_{i,k}\|^2)
	=\frac{1}{2}N\|\nabla f(\ox_k)\|^2+\frac{L^2 N}{2}\oV_k.
	\end{multline*}
	From Assumption \ref{asp: gradient samples}, (\ref{nabla+a_2}) and the above inequality, we obtain from (\ref{E f(ox_k+1) pre}) that
	\begin{multline*}
	\mathbb{E}[f(\ox_{k+1})\mid \mx_k] 
	\le   f(\ox_k)-\frac{\tg}{2N}\|\nabla f(\ox_k)\|^2+\frac{L^2\tg}{2N}\oV_k\\
	+\frac{L\tg^2}{N^2}\left[2\Vert\nabla f(\ox_k)\Vert^2+(2L^2+4a^2\overline{d}^2)\oV_k\right]+\frac{L\tg^2}{2N^2}\sigma^2.
	\end{multline*}
	By Lemma \ref{lem2}, Assumption \ref{asp:Lipschitz} and (\ref{nabla+a_2}),
	\begin{align*}
	\mathbb{E}[\oV_{k+1}|\mathbf{x}_k] \le & \oV_k-\frac{2\tg }{N^2}\sum_{i=1}^{N}[\nabla^T f(x_{i,k})-\nabla^T f(\ox_k)]e_{i,k}-\frac{2}{N}a\lambda_2\tg \oV_k\\
	& +\frac{2\tg^2}{N}\left[2\Vert\nabla f(\ox_k)\Vert^2+(2L^2+4a^2\overline{d}^2)\oV_k\right]+\frac{\tg^2\sigma^2}{N}\\
	\le & \oV_k+\frac{2}{N}(L-a\lambda_2)\tg \oV_k
	+\frac{2\tg^2}{N}\left[2\Vert\nabla f(\ox_k)\Vert^2+(2L^2+4a^2\overline{d}^2)\oV_k\right]\\
	& +\frac{\tg^2\sigma^2}{N}.
	\end{align*}
	Note the definition of $\check{\omega}$ in (\ref{check_omega}); the condition $a>5L/(4\lambda_2)$ implies $4N(a\lambda_2-L)>NL$, and thus when $\gamma$ is small enough, we have $\check{\omega}\in(0,1)$. Then from the last two relations we obtain that
	\begin{multline*}
	\mathbb{E}[f(\ox_{k+1}) +\check{\omega}L\oV_{k+1}|\mathbf{x}_k]
	\le f(\ox_k)+\check{\omega}L\oV_k-\frac{\tg}{N}\left(\frac{1}{2}-\frac{2L\tg}{N}-4\check{\omega} L\tg\right)\|\nabla f(\ox_k)\|^2\\
	+\left[\frac{L^2}{2N}+\frac{2\check{\omega} L}{N}(L-a\lambda_2)\right]\tg\oV_k+\left(\frac{L}{N^2}+\frac{2\check{\omega} L}{N}\right)(2L^2+4a^2\overline{d}^2)\tg^2\oV_k\\
	+\left(\frac{1}{2N}+\check{\omega}\right)\frac{L\tg^2\sigma^2}{N}.
	\end{multline*}
	By (\ref{check_omega}), $\check{\omega}$ satisfies
	\begin{equation*}
	\frac{L^2}{2N}+\frac{2\check{\omega}L}{N}(L-a\lambda_2)+\left(\frac{L}{N^2}+\frac{2\check{\omega}L}{N}\right)(2L^2+4a^2\overline{d}^2)\tg=0.
	\end{equation*}
	It follows that
	\begin{multline*}
	\frac{\tg}{N}\left(\frac{1}{2}-\frac{2L\tg}{N}-4\check{\omega}L\tg\right)\|\nabla f(\ox_k)\|^2\le f(\ox_k)+\check{\omega}L\oV_k-\left[\mathbb{E}[f(\ox_{k+1}) +\check{\omega}L\oV_{k+1}|\mathbf{x}_k]\right]\\
	+\left(\frac{1}{2N}+\check{\omega}\right)\frac{L\tg^2\sigma^2}{N}.
	\end{multline*}
	Noting the definition of $\check{\mu}$ in (\ref{check_mu}), taking full expectation and summing from $k=0$ to $k=K-1$ on both sides of the above inequality, we have
	\begin{equation*}
	\check{\mu}N\sum_{k=0}^{K-1}\|\nabla f(\ox_k)\|^2\le f(\ox_0)-f^*+\check{\omega}L\oV_0+\left(\frac{1}{2N}+\check{\omega}\right)\frac{KL\tg^2\sigma^2}{N}.
	\end{equation*}
	Given that $\mathbb{P}[R=k]=1/K,\forall k=0,1,\ldots,K-1$, we conclude that
	\begin{multline*}
	\frac{1}{L}\mathbb{E}[\|\nabla f(\ox_R)\|^2] = \frac{1}{KL}\sum_{k=0}^{K-1} \mathbb{E}[\|\nabla f(\ox_k)\|^2]\\
	\le \frac{1}{NK\check{\mu}}\left[\frac{f(\ox_0)-f^*+\check{\omega}L\oV_0}{L}+\left(\frac{1}{2}+\check{\omega}N\right)\frac{K\tg^2\sigma^2}{N^2}\right].
	\end{multline*}

\bibliographystyle{siamplain}
\bibliography{mybib}

\end{document}